\documentclass[a4paper,11pt]{article}

\usepackage{fullpage}
\usepackage[utf8]{inputenc}
\usepackage{comment}
\usepackage[english]{babel}
\usepackage{bm}
\usepackage{multirow,hhline,colortbl}
\usepackage{dsfont}
\usepackage{eurosym}

\renewcommand{\geq}{\geqslant}
\renewcommand{\leq}{\leqslant}

\usepackage{algorithm,algorithmicx}
\usepackage{algpseudocode}

\usepackage{empheq}    
\usepackage{amsthm}
\usepackage{graphicx}
\usepackage[justification=centering]{caption}
\usepackage{subcaption}
\usepackage{amssymb,mathtools,epsfig}
\usepackage{enumitem}

\usepackage[customcolors]{hf-tikz}

\tikzset{style green/.style={
    set fill color=green!50!lime!60,
    set border color=white,
  },
  style cyan/.style={
    set fill color=cyan!90!blue!60,
    set border color=white,
  },
  style orange/.style={
    set fill color=orange!80!red!60,
    set border color=white,
  },
  hor/.style={
    above left offset={-0.15,0.31},
    below right offset={0.15,-0.125},
    #1
  },
  ver/.style={
    above left offset={-0.1,0.3},
    below right offset={0.15,-0.15},
    #1
  }
}

\usepackage{hyperref}
\usepackage[capitalise]{cleveref}
\usepackage{footnote}
\makesavenoteenv{tabular}
\makesavenoteenv{table}

\DeclareMathOperator*{\cl}{cl}
\DeclareMathOperator*{\st}{s.t.}
\DeclareMathOperator*{\Vertices}{Vert}
\DeclareMathOperator*{\Int}{Int}
\DeclareMathOperator{\Proj}{Proj}
\DeclareMathOperator*{\graph}{gph}

\DeclareMathOperator*{\argmin}{arg\,min}
\DeclareMathOperator*{\val}{\textit{v}\,}

\DeclareMathOperator{\bbR}{\mathbb{R}}
\DeclareMathOperator{\bbN}{\mathbb{N}}

\DeclareMathOperator{\trt}{\!\text{-}\!}
\DeclareMathOperator{\calW}{\mathcal{W}}
\newcommand{\shortminus}{\scalebox{0.5}[1.0]{$-$}}

\newtheorem{prop}{Proposition}[section]
\newtheorem{theorem}[prop]{Theorem}
\newtheorem{defin}[prop]{Definition}
\newtheorem{corol}[prop]{Corollary}
\newtheorem{lemma}[prop]{Lemma}

\crefname{lemma}{lemma}{lemmas}
\Crefname{lemma}{Lemma}{Lemmas}
\crefname{theorem}{theorem}{theorems}
\Crefname{theorem}{Theorem}{Theorems}
\crefname{prop}{proposition}{propositions}
\Crefname{prop}{Proposition}{Propositions}
\crefname{defin}{definition}{definitions}
\Crefname{defin}{Definition}{Definitions}

\newtheorem{hypo}{Assumption}[section]

\theoremstyle{remark}
\newtheorem*{remark}{Remark}

\DeclareMathSymbol{\mlq}{\mathord}{operators}{``}
\DeclareMathSymbol{\mrq}{\mathord}{operators}{`'}

\makeatletter
\newcommand*\frob{\mathpalette\bigcdot@{.7}}
\newcommand*\bigcdot@[2]{\mathbin{\vcenter{\hbox{\scalebox{#2}{$\m@th#1\bullet$}}}}}
\makeatother

\usepackage{setspace}
\usepackage{apacite}
\begin{document}

\renewcommand\labelitemi{$\diamond$}

\title{\textbf{Quadratic Regularization of Bilevel Pricing Problems and Application to Electricity Retail Markets}}

\author{Quentin Jacquet$^{1,2}$\quad Wim van Ackooij$^1$ \quad Cl\'emence Alasseur$^1$ \quad St\'ephane Gaubert$^2$\\[5pt]
\small $^1$ EDF Lab Saclay, Palaiseau, France \\\small\texttt{\{quentin.jacquet, wim.van-ackooij,clemence.alasseur\}@edf.fr} \\
\small$^2$INRIA, CMAP, Ecole polytechnique, IP Paris, CNRS, Palaiseau, France\\ \small\texttt{stephane.gaubert@inria.fr}
}

\date{} 
\maketitle\thispagestyle{empty} 

\begin{abstract}We consider the profit-maximization problem solved by an electricity retailer who aims at designing a menu of contracts. This is an extension of the unit-demand envy-free pricing problem: customers aim to choose a contract maximizing their utility based on a reservation bill and multiple price coefficients (attributes). A basic approach supposes that the customers have deterministic utilities; then, the response of each customer is highly sensitive to price since it concentrates on the best offer. A second classical approach is to consider logit model to add a probabilistic behavior in the customers' choices. To circumvent the intrinsic instability of the former and the resolution difficulties of the latter, we introduce a quadratically regularized
  model of customer's response,
  which leads to a quadratic program under complementarity constraints (QPCC).    
  This allows to robustify the deterministic model, while keeping a strong geometrical structure. In particular, we show that the customer's response is governed by a polyhedral
  complex, in which every polyhedral cell determines a set of contracts which is effectively chosen. Moreover, the deterministic model is recovered as a limit case of the regularized one. We exploit these geometrical properties to develop a pivoting heuristic, which we compare with implicit or non-linear methods from bilevel
  programming, showing the effectiveness of the approach.
  Throughout the paper, the electricity retailer problem is our guideline, and we present a numerical study on this application case.
\end{abstract}

\textbf{Keywords:} Pricing, Bilevel problem, Polyhedral complex, Logit, Electricity contracts.

\section{Introduction}

\subsection{Context}
For a company, the question of determining the correct prices of its products is crucial: a compromise has to be found between having enough consumers buying products and setting prices that are sufficiently important to cover the production cost. Profit-maximization models have been extensively studied. They consist in maximizing the seller profit taking in account the customer behavior.
The special structure of these problems can be generally cast into the bilevel framework, see~\shortciteA{Bard_2013} and~\shortciteA{dempe_2015}.
In this setting, a \emph{leader} (here the company) aims at optimizing its own objective (\emph{upper level}), taking into account the decision of the \emph{follower} (here the consumers), obtained as the solution of an inner optimization problem (\emph{lower level}). This 2-player problem is known in game theory as a Stackelberg game, see~\shortciteA{Stackelberg_1952}, and reflects the asymmetry of the players' roles: the leader moves first, then the follower replies (sequential games).
As detailed in~\shortciteA{Kleinert_2021}, two classical approaches consist in reformulating the problem as a single-level one, either using strong-duality or the KKT conditions, to express the optimality of the lower decision
and constrain the upper problem. Formulations based on KKT conditions lead to Mathematical Programs with Complementarity Constraints (MPCC), a class of optimization problems whose interest has been growing in recent years, and particularly in the energy sector, see~\shortciteA{Afcsar_2016, Alekseeva_2019, Aussel_2020, Abate_2021}.

The \emph{unit-demand envy-free} pricing problem is a specific case. We consider a finite number of customers (or \emph{segments} of customers) who are supposed to buy precisely one product, among the ones maximizing their utility. Moreover, products are available in unlimited supply. \shortciteA{Guruswami_2005} showed that this problem is APX-hard (even on a restricted class of instances).
\shortciteA{Shioda_2011} developed Mixed-Integer Programming (MIP) formulations, along with valid cuts and heuristics. They also enhance the model to ensure that each customer faces a unique maximum utility. \shortciteA{Fernandes_2016} compare several MIP formulations and reinforce them with new valid cuts. All these approaches are based on deterministic models of customer's response. By their deterministic nature, they lead to instability features: the customer's response is discontinuous, resulting in typical ``sawtooth'' shaped profit functions, see e.g.~\shortciteA{Labbe_1998,Gilbert_2015} or~\Cref{fig::comparison_regularization} below. 

There are situations in which revenue management data are uncertain, and as noted in~\shortciteA{Tuncel_2008}, it is desirable that ``optimal or near-optimal prices 
delivered by the optimization techniques [be] robust under modest perturbations of 
the reservation prices of the potential customers and the competitors’ prices".
This question of uncertain data and/or uncertain decision of the followers is nowadays a central question in the bilevel community, see e.g.,~the recent survey~\shortciteA{Beck_2023}.
To overcome instability issues, one approach is to consider choice models of a probabilistic nature. Then, the value of the lower level objective determines the probability distribution
of the customer's choice. The most studied case concerns the \emph{logit} model, see \shortciteA{McFadden_1974,Train_2009}.
\shortciteA{Li_2011} suppose that the population is homogeneous,
meaning that there is only one segment. They reformulate the problem as a concave maximization problem by a \emph{market-share} transformation.
\shortciteA{Shao_2020} extend this approach to the case of multiple price attributes. Logit pricing models with multiple consumers segments have only been studied very recently: \shortciteA{Li_2019} formalize the pricing problem under the Mixed Multinomial Logit (MMNL), and develop algorithms to find good solutions. \shortciteA{Hohberger_2020} applies such models to the revenue management case study of the German long-distance railway network.
However, logit-based models are in general hard to solve with guarantees of optimality, owing to their nonlinear and nonconvex nature.

\subsection{Contribution}
We consider a \emph{multi-attribute} version of the {\em unit-demand envy-free pricing} problem that we model by a bilinear bilevel formulation.  This applies
in particular to the pricing of electricity offers, which is our driving
case study.



Our main contribution is the development of a new model, based on a {\em quadratic regularization} of customer's response: it has the same benefits
as the logit-based models in terms of realism and robustness, whereas
its quadratic nature allows one to apply efficient algorithms
based on polyhedral geometry.



First, we give a closed-form expression of the lower response and highlight its polyhedral structure (\Cref{prop::poly_complex}). This shows in particular that in the presence
of near ties (contracts with similar utilities), customer's
response distributes among the best contracts, rather than concentrating on
a single one. The regularization parameter measures the ``rationality''
of the customer, in particular, the deterministic response
is recovered as a limit case -- with perfectly rational customers (\Cref{theorem::asymptotic_cells}).
More precisely, we show that in the regularized model, the
response is governed by a polyhedral complex, in which each open cell determines a set of contracts which are effectively chosen.

Besides, we show that this model
has the same good theoretical properties as the logit model (stability)
and provide metric estimates showing that the responses of the two models
are close (\Cref{app::estimates}).
The main interest of quadratic regularization,
then, lies in computational tractability. We show that the regularized bilevel
model reduces to a convex Quadratic Program with Complementary
Constraints (QPCC). Powerful methods based on mixed or semidefinite programming
allow one to solve instances of significant size of QPCC with
optimality guarantees, although problems of this kind are generally difficult.
In fact, we show in \Cref{sec-complexity} that solving the present quadratic model is APX-Hard, by reusing the transformation introduced for the deterministic case in~\shortciteA{Guruswami_2005}.
We develop in \Cref{sec-localsearch} an efficient local search method, QSPC (Quadratic Search on the Price Complex), exploiting the polyhedral structure of the customer's
response. 

Finally, we consider realistic instances arising from French electricity markets, and analyze the optimal solution in both deterministic and quadratic cases. In particular, we look at the customers' distribution to illustrate the influence of a regularized lower level. A performance comparison between the proposed algorithm and other methods from the literature is also given in~\Cref{app::comparison_solvers}.


Our study is inspired by several works. We adopt the viewpoint of~\shortciteA{Gilbert_2015} in that we consider the MMNL model~\shortcite{Li_2019, Hohberger_2020} as a regularized version of its deterministic analog~\shortcite{Shioda_2011,Fernandes_2016}. They look at a related problem that studies the toll pricing optimization, and demonstrate, among other things, asymptotic convergence of the logit regularization to the deterministic model. 
Besides, \shortciteA{Shioda_2007} introduced several probabilistic choice models as alternatives to the logit approach, and developed convex mixed-integer formulations to solve them. In particular, they considered a model which depends on the surplus of the products, and we design a new customer's response that satisfies this assumption.
By comparison with all these works, the main novelty is the introduction of the quadratic regularized model as a new probabilistic customer's response and the evidences that it has the same good features as the logit model, in terms of economic realism and robustness, while being computationally more tractable. 
\shortciteA{Dempe_2001} also inverstigate quadratic regularization on bilinear bilevel problems and develop bundle trust region algorithm to solve them. We differ from their work by specializing the lower level to be defined on the simplex, and by describing the customers' choices as a polyhedral complex. This interpretation is inspired by the study of~\shortciteA{Baldwin_2019}, who showed that for deterministic models, agent's response can be represented by a polyhedral complex,  a tropical hypersurface. This tropical complex is recovered as a limit case of the present polyhedral complex when the
regularization term vanishes. 

The paper is organized as follows. In \Cref{sec-prelim}, we present
the deterministic multi-attribute unit-demand envy-free pricing problem and establish basic properties of the model (optimality of integer low-level solutions, reformulation as a single level problem using the KKT conditions). For comparison, we also recall the definition of the logit-based model.
In \Cref{sec-quadratic}, we introduce the quadratically regularized model,
in particular, we describe the geometric properties of customer's
response, and provide a reformulation as a single level
QPCC. In \Cref{sec-localsearch}, we develop the local search method (QSPC),
exploiting the polyhedral structure of customer's response.
In \Cref{sec-numerical}, we provide a numerical analysis on  instances from the electricity pricing problem.

\section{Preliminaries}\label{sec-prelim}
\subsection{Notation}
In the sequel, we denote by $\Delta_{N}$ the simplex of $\bbR^N$, and
by $\|x\|_{N}$ the Euclidean norm associated with the canonical scalar product $\left<x,y\right>_N$ on $\bbR^N$.
For any polyhedron $Q$, $\Vertices(Q)$ denotes the set of vertices of $Q$.
Moreover, for any optimization problem $(P)$, the value $\val(P) \in \bbR \cup \{\pm \infty\}$ denotes its optimal value (that can be infinite if $(P)$ is infeasible or unbounded).

\subsection{Deterministic model}
We suppose that a company has $W$ different types of contracts and that a market study has distinguished beforehand $S$ customers segments, each of them gathering consumers that have approximately the same behavior.
Given a segment $s \in [S]:=\{1,\dots,S\}$ and a product $w\in [W]$,
the \emph{reservation bill} $R_{sw}$ is the maximum bill that customers of
this segment are willing to pay on $w$. 
In the classical product pricing model, the items to sell are only characterized by a price (determined by the company) and each customer faces the same price. In our setting, we consider the \emph{multi-attribute} case where the bill of each contract $w$ is determined by a finite number $H>1$ of variables (or attributes), denoted by $x^h_w$. For instance,  in the French electricity market,
the invoice of a customer depends on at least two variables, representing a fixed and a variable component, the former depending on the subscribed power of the customer and the latter depending on his electricity consumption, see~\shortciteA{CER_2004}. Moreover, in the peak/off-peak contract, the variable component distinguishes between the peak 
and off-peak consumption. Then, the invoice is determined
by at least three variables. The following assumption
captures such contracts. 
\begin{hypo}
The bill $\theta_{sw}(x)$ paid by segment $s$ for contract $w$ is a linear form:
\begin{equation}
  \label{eq::def-theta}
  \theta_{sw}(x) := \left<E_{sw}, x_w\right>_H\enspace ,
  \end{equation}
where $E_{sw} =(E_{sw}^h)_{h\in H}\in \bbR_{\geq 0}^H$.
Besides, the price coefficients $x^h_w$ are supposed to be in a non-empty polytope $X\subset \bbR^{W\times H}$.
\label{hypo::on_X}
\end{hypo}
In the electricity market context, $E_{sw}$ represents the electricity consumption of the customers of segment $s$ who choose the contract $w$.
It depends on $h$ (the period of the day)
and on the contract type $w$. This is realistic, since the notion of peak and off-peak period can vary along the contracts, and since customers adapt their electricity consumption depending on their choice of contract.
Note that, in this model, the consumption does not depend on the price. This (strong) simplification is justified by a high inelasticity of the electricity demand in the short run, see e.g.,~\shortciteA{Csereklyei_2020}. Hence, this model constitutes a first-order model, and aims at focusing on the uncertainty of the decision, an active research field in bilevel programming~\shortcite{Beck_2023}. Here, we are not looking at long-term policies, but we focus on finding the best price policy to a given set of competitors' offers at a given time.
The situation in which the bill $\theta_{sw}(x)$
is \emph{affine} in the energy consumption, instead
of being linear as in~\eqref{eq::def-theta}, reduces to the latter
case by adding to the set $H$ an extra element $h=0$, with
$E_{sw}^0=1$ for all $s,w$. 
This is the case here, where \eqref{eq::def-theta} simultaneously takes into account the fixed part (contracted power) and the variable portion (depending on the consumption). We also make classical assumptions:
\begin{hypo}\begin{enumerate}[label=(\roman*)]
\setlength\itemsep{-0.1\baselineskip}
\item \emph{Unit-Demand}: Each customer purchases exactly one contract.
\item \emph{Envy-free}: There is no limitation on the number of customers able to purchase the same contract and so each customer chooses a contract maximizing his utility.
\item \emph{No-purchase option}: Consumers have the option not to purchase any contract, or in a competitive environment, to choose a contract from a competitor. 
\end{enumerate}
\end{hypo}

The \emph{utility} of segment $s$ for contract $w$ is the difference between the reservation bill and the invoice, i.e., $$U_{sw}(x) := R_{sw} - \theta_{sw}(x)\enspace.$$ 
The \emph{disutility} is then the opposite of the utility. 
The no-purchase option corresponds to the fact that in competitive
environment, customers can choose a contract among those proposed by  competitors. We
assume here that the competition is static, meaning that competitors do not react to the company prices. Therefore, competing contracts could be understood in our context as regulated alternatives (for instance, in the French electricity market, there are several such offers with prices determined by a regulation authority). More generally, the contracts from different static competitors
can be aggregated in a unique contract of a virtual competitor, and the reservation bill $R_{sw}$ consists here in the infimum of the bills proposed by the competition to segment $s$ (there can be an additional term representing a given preference for the contract $w$). This utility is also called \emph{surplus}, as it corresponds to the additional gain in terms of utility that a consumer can expect by choosing an offer from the leader, compared to the no-purchase option. The utility of the no-purchase option is therefore set to be $0$.
\begin{remark}We could also set the utility to be the opposite of the bill, i.e., $U_{sw}(x) = -\theta_{sw}(x)$ and the no-purchase utility to be $R_{sw}$, but as the utilities are defined up to an additive constant in choice models, the standard normalization is to set the no-purchase utility to 0. 
\end{remark}

Finally, when a segment $s$ chooses a contract $w$, the company has to fulfill the service, implying a cost $C_{sw}$. In the case of an electricity retailer, it has to supply electricity.

To model the customers behavior of segment $s$, we define the variables $y_{s} \in \bbR^W$ such that
\begin{align} \forall s \in [S],\, w \in [W],\quad y_{sw} = \begin{cases} 1 \text{ if segment $s$ chooses $w$,}\\ 0 \text{ otherwise.}\end{cases}
\label{eq::def-behavior}
\end{align}
To make explicit the no-purchase option, we introduce a variable $y_{s0}$ and denote the extended choice vector  for segment $s$ by
$\bar{y}_s  := (y_{s0},y_s) \in \bbR\times \bbR^W = \bbR^{W+1}.
$
We shall think of an element $\bar{y}_s \in \Delta_{W+1}$ as a {\em relaxed} choice of segment
$s$. When $\bar{y}_s$ is  a vertex of $\Delta_{W+1}$, $y_s=(y_{sw})_{w\in W}$
determines the behavior of segment $s$, according to~\eqref{eq::def-behavior}.
The no-purchase option corresponds to $y_{s0}=1$.
For a price strategy $x\in X$, the customers behavior is defined by the solution set mapping $\Psi$ defined as
\begin{equation}
\Psi(x):= \argmin_{\bar{y}' \in (\Delta_{W+1})^S} \left\{\sum_{s\in [S]}\left<\theta_s(x)-R_s, y'_s\right>_W \right\} \enspace .
\label{eq::o_lower_problem}
\end{equation}
Note that the scalar product that appears in the objective is on $\bbR^W$ since the no-purchase option induces a zero utility for any customer.

The \emph{multi-attribute unit-demand envy-free pricing problem} can now be expressed as the following bilinear bilevel model
\begin{equation}
\tag{$o\trt BP$}
\max_{x \in X,\bar{y}} \;\left\{\left.F(x,\bar{y}):=\sum_{s\in [S]} \rho_s \left<\theta_s(x) - C_s, y_s\right>_W \;\right\vert\;(x,\bar{y}) \in  \graph \Psi\right\}
\label{eq::o_BPR_choices}
\end{equation}
In the model, $\rho_s$ stands for the weight of segment $s$ in terms of company's profit. 
Note that there is an asymmetry in the two objective functions: the leader aims at maximizing quantities $\theta_{sw}(x) - C_{sw}$ while the follower aims at minimizing the disutility $\theta_{sw}(x) - R_{sw}$. The very special case $C= R$ would lead to a subclass of bilevel problems, known as zero-sum games, see e.g.,~\shortciteA{Washburn_2014}.
The label $(o\trt BP)$ refers to the
{\em optimistic} nature of this bilevel problem: if the lower level problem
has several optimal solutions, the upper level optimizer takes into account
the most favorable of these optimal solutions, see e.g.~\shortciteA{dempe_2015}.
\begin{remark} Because all the segments react \emph{independently}, we can aggregate all their actions under the same problem. Hence, the minimization
  in the lower level problem~\eqref{eq::o_lower_problem} is made
  over the Cartesian product of simplices. The vertices of each of these
  simplices represent the possible decisions of a given segment.
\end{remark}

The following result justifies the minimization over relaxed choices
in~\eqref{eq::o_BPR_choices}.
\begin{prop}
There exists an optimal solution of \eqref{eq::o_BPR_choices} with integer lower values $y$.
\label{prop::relaxation_BP}
\end{prop}
\begin{proof}
  We denote by $(x^*,\bar{y}^*)$ an optimal solution, which exists because $\graph \Psi$ is
  compact and non-empty (from \Cref{hypo::on_X}).
  The argmin set $\Psi(x^*)$ is a face of the Cartesian product
  of simplices $(\Delta_{W+1})^S$, since it arises from the minimization of a linear objective on this product. So it is a non-empty integer polyhedron.
  Moreover, there exists an extreme point of $\Psi(x^*)$, denoted by $\hat{y}$, such that $\displaystyle F(x^*,\hat{y})= F(x^*,\bar{y}^*)$ owing to the linearity in $y$ of the upper objective. To conclude, $(x^*,\hat{y})$ is also an optimal solution and
$\hat{y}$ is integer as extreme point of $\Psi(x^*)$.
\end{proof}
Problem~\eqref{eq::o_BPR_choices} is a very specific bilinear bilevel problem with a quite simple lower problem (minimization over the simplex, without integrity constraints). However, despite its apparent simplicity, this model is APX-hard since it includes as a special case the unit-demand envy-free pricing model, which was shown to be APX-hard, see~\shortciteA{Guruswami_2005}.

The problem $(o \trt BP)$ is a \emph{profit-maximization} problem: in fact, we can define the optimistic leader profit function $\pi^{opt}$ for a given price strategy $x$ as 
\begin{equation}
\pi^{opt}(x) := \sum_{s\in [S]} \rho_s \sum_{w\in [W]} (\theta_{sw}(x) - C_{sw}) y^{opt}_{sw}(x)
\label{eq::det_profit_function}
\end{equation}
where $y^{opt}_{sw}(x)$ is the optimistic lower response (which is binary, see \eqref{eq::def-behavior}). The problem $(o \trt BP)$ is therefore the maximization of the function $\pi^{opt}$ over $X$.
The optimistic profit function $\pi^{opt}$ is piecewise linear (the profit is linear for a given customers distribution $y$, and the possible customers distribution lies in a discrete set).
However, $\pi^{opt}$ is in general discontinuous at prices inducing \emph{ties} (multiple minimum disutilities for a segment), see~\cref{fig::comparison_regularization}.

The most common way to express the optimality of the lower problem as a system of inequalities is to use the Karush-Kuhn-Tucker (KKT) conditions. Applying this idea to \eqref{eq::o_BPR_choices} leads to the following formulation
\begin{equation}\label{eq::o_LPCC}
\tag{$o \trt KKT$}\begin{aligned}
\displaystyle \max_{x\in X,\bar{y}} &\quad \sum_{s\in[S]} \rho_s\mu_s + \rho_s\left<R_s-C_s,y_s\right>_W\\
~
\st &\quad 0 \leq y_{sw} \perp \theta_{sw}(x) - R_{sw} - \mu_s \geq 0,\,\forall s,w\\
&\quad 0 \leq y_{s0} \perp  \mu_s \leq 0,\,\forall s\\
&\quad \bar{y}_s \in \Delta_{W+1},\,\forall s
\end{aligned}
\end{equation}
%

To numerically solve this formulation, we usually replace the complementarity constraints by Big-$M$ constraints introducing new binary variables. Using \Cref{prop::relaxation_BP}, we provides a compact formulation in which the lower variables $y_s$ are the only binary variables:
\begin{equation}
\begin{aligned}
\displaystyle \max_{x\in X,\bar{y}} &\quad \sum_{s\in[S]} \rho_s\mu_s + \rho_s\left<R_s-C_s,y_s\right>_W\\
~
\st &\quad 0 \leq \theta_{sw}(x) - R_{sw} - \mu_s \leq M_{sw}(1-y_{sw}),\,\forall s,w\\
&\quad 0 \leq -\mu_s \leq M_{s0}(1-y_{s0}),\,\forall s\\
&\quad \bar{y}_s \in \Vertices(\Delta_{W+1}),\,\forall s
\end{aligned}
\label{eq::o_bigM}
\end{equation}
Here, the set of vertices $\Vertices(\Delta_{W+1})$ is known and is equal to $\left\{y \in \{0,1\}^{W+1} \,\vert \,\sum_{w = 0}^W y_w = 1\right\} $.
The Big-$M$ parameters $M_{sw}>0$ must be chosen to be sufficiently large to prevent the elimination of any optimal solution, see~\shortciteA{Pineda_2019,Kleinert_2023}. This is in general as  hard as solving the initial bilevel problem, see~\shortciteA{Kleinert_2020bis}. However, in the present case,
owing to the boundedness of the pricing variables $x\in X$ and the structure of the constraints, we can explicitly find valid  Big-$M$ values. If $X\subseteq \prod_{1\leq w\leq W}[x^-_w,x^+_w]$, then it sufficies to take:
$$
M_{sw} = \theta_{sw}(x^+) - R_{sw} + M_{s0}\,,\quad M_{s0} =  \max\{0,\max_{1\leq w \leq W} \left\{R_{sw} - \theta_{sw}(x^-) \right\}$$

\begin{remark}
The formulations~\eqref{eq::o_LPCC} and~\eqref{eq::o_bigM} generalize the (U) formulation introduced by~\shortciteA{Fernandes_2013} that applies in the single-attribute case: the variables $\mu_s$ express the disutilities of each segment $s$. 
\end{remark}

\subsection{Logit regularization}

The formulation \eqref{eq::o_BPR_choices} models customers reactions as deterministic behaviors. It relies on two assumptions: \begin{enumerate}[label=(\roman*)]
\setlength\itemsep{-0.1\baselineskip}
\item customers have perfect rational and deterministic behavior,
\item parameters such as reservation bills and costs are perfectly known.
\end{enumerate}  
Both assumptions can be discussed: not only real customers are not purely rational agents in that they can choose a contract that does not maximize the utility, but also a segment is the aggregation of quasi-similar customers, not strictly identical ones. Therefore in reality, when a segment faces two very close disutilities, customers of this segment are likely to spread themselves over the two possibilities. Besides, the reservation bills and costs are estimations obtained by analysis on the market but cannot be known exactly. Hence, assuming lower response to be binary as in the optimistic model can be quite unrealistic and may lead to an unachievable optimum. This can be avoided by \emph{Logit} modeling which captures the probabilistic nature of customers' choice by adding a Gumbel uncertainty. There is a wide literature which uses this approach as choice models, see e.g.~\shortciteA{Train_2009} and the references therein.

Previously, consumers were supposed to choose a contract minimizing their deterministic disutility i.e., each segment $s\in [S]$ selects $w^*\in \{0\hdots W\}$ such as $V_{sw^*} = \min_{w \in \{0\hdots W\} }V_{sw}$ where $V_{sw} := \theta_{sw}(x)-R_{sw}$ for all $w\in[W]$ and $V_{s0} := 0$. We now suppose that their disutilities are defined as $$U_{sw} := \beta V_{sw} + \varepsilon_{sw},\; \forall s,w,$$ 

where $\{\varepsilon_{sw}\}_w$ is a family of Gumbel random variables, distributed identically and independently,
and $\beta \geq 0$ is an inverse temperature in the sense of physics. 
The choice of Gumbel uncertainties is standard in discrete choice theory, and the main underlying assumption is not so much about the shape of the uncertainty but rather on the independence of the noises~\shortcite[Chapter 3]{Train_2009}. Here, we suppose that the utilities capture enough information so that the remaining part of the uncertainty behaves as a white noise.

\begin{remark}
In the sequel, we consider a common $\beta$ across the segments, but all the results still apply for a differentiated value $\beta_s = d_s \beta$, where $d_s$ is a given parameter. This corresponds to a rescaling of $\beta$, adapted to each segment.
\end{remark}
As a consequence, the lower response is expressed as 
\begin{equation}y_{sw} = \mathbb{P}[U_{sw}\leq U_{sw'}, \,\forall w' \neq w ],\, \forall s,w \enspace .
\label{eq::prob_utility}
\end{equation}
Hence, a customer has a probability to choose a contract which is not the optimal one in terms of deterministic utility. 
The calculation of the probability $y_{sw}$ arising in equation \eqref{eq::prob_utility} is done in \shortciteA{Train_2009} and it has an explicit form. Replacing the deterministic lower response by this expression of $y_{sw}$ leads to the following \emph{Mixed Multinomial Logit} model:
\begin{equation}\label{eq::MMNL_model}
\tag{$\beta\trt BP$}
\begin{aligned}
\max_{x \in X,y} &\quad \sum_{s\in [S]} \rho_s \left<\theta_s(x) - C_s, y_s\right>_W \\
~
\st &\quad y_{sw} = \frac{e^{-\beta (\theta_{sw}(x)-R_{sw})}}{1+\sum_{w'\in [W]} e^{-\beta (\theta_{sw'}(x)-R_{sw'})}},\,\forall s,w
\end{aligned}
\end{equation}

\begin{remark}
The '1' in the denominator corresponds to the no-purchase option.
\end{remark}

Equivalently, we recall here a standard reformulation of \eqref{eq::MMNL_model} :
\begin{prop}
Problem \eqref{eq::MMNL_model} is equivalent to 
\begin{equation}
\begin{aligned}
\max_{x \in X,\bar{y}} &\quad \sum_{s\in [S]} \rho_s \left<\theta_s(x) - C_s, y_s\right>_W \\
~
\st &\quad \bar{y}_s \in  \argmin_{\bar{y}_s' \in \Delta_{W+1}} \left\{\left<\theta_s(x) - R_s, y'_s\right>_W + \frac{1}{\beta} \left<\log(\bar{y}'_{s}), \bar{y}'_{s}\right>_{W+1} \right\},\,\forall s
\end{aligned}
\label{eq::bilevel_logit}
\end{equation}
\end{prop}
\begin{proof}
    Given $V\in\bbR^W$, we study the problem: $\min_{\bar{y} \in \Delta_{W+1}} \left\{\left<V, y\right>_W + \beta^{\shortminus 1} \left<\log(\bar{y}), \bar{y}\right>_{W+1} \right\}$.
First note that the positivity assumption is always satisfied at the optimum, since the function $y\log(y)$ acts a barrier. Looking at the KKT optimality conditions , we then obtain that there exists $\mu\in\bbR$ (dual variable of the constraint $\sum_w y_w = 1$) such that for any $w\leq W$, $0 = V_w + \tfrac{1}{\beta} (\log(y_w) +1) - \mu$. This implies that $y_{w} = \exp(\beta\mu - 1)\exp(-\beta V_w)$. As $\bar{y}$ must lie in the simplex, we recover the standard expression of the logit model.
\end{proof}

This model highlights that the logit expression is the optimum of a strictly convex minimization problem (the property was pointed out in \shortciteA{Fisk_1980} and \shortciteA{Gilbert_2015}). The objective function is the deterministic one function  to which we add
the entropic
regularization term
$\beta^{\shortminus 1}\left<\log(\bar{y}'_{s}), \bar{y}'_{s}\right>_{W+1}$,
attracting the lower response to the center of the simplex $\Delta_{W+1}$.

The model~\eqref{eq::MMNL_model} is intrinsically defined as a single-level problem since the lower response for any segment $s$ is unique and analytically known. For a given price strategy $x$, we define the leader profit function $\pi^{log}(x;\beta)$ as 
\begin{equation}
\pi^{log}(x;\beta):= \sum_{s\in [S]} \rho_s \sum_{w \in [W]}(\theta_{sw}(x)-C_{sw})y^{log}_{sw}(x;\beta)
\label{eq::pi_logit}
\end{equation}
where $y^{log}$ stands for the logit lower response.
This objective function $\pi^{log}$ is in general neither concave nor convex, see~\shortciteA{Li_2019}.

\section{Quadratic regularization}\label{sec-quadratic}

In the case of a homogeneous population and unconstrained prices, \shortciteA{Li_2019} express the problem~\eqref{eq::MMNL_model} in terms of lower variables to obtain a concave maximization problem. If we add bounds on prices and consider multi-attribute utilities, \shortciteA{Shao_2020} show another concave transformation that keeps tractability in the resolution. However, with heterogeneous segments as it is the case here, no tractable transformation is known, and only local optimum of \eqref{eq::MMNL_model} can generally be found. This motivates us to look at a new convex penalization, replacing the entropy penalization term in \eqref{eq::bilevel_logit} by a quadratic one.
\begin{equation}
\tag{$q\beta \trt BP$}
\begin{aligned}
\max_{x \in X,y} &\; \sum_{s\in [S]} \rho_s \left<\theta_s(x) - C_s, y_s\right>_W \\
~
\st &\; \bar{y}_s \in  \argmin_{\bar{y}_s' \in \Delta_{W+1}} \left\{\left<\theta_s(x) - R_s, y'_s\right>_W + \frac{1}{\beta}\left<\bar{y}_s'-1, \bar{y}_{s}'\right>_{W+1}\right\},\,\forall s
\end{aligned}
\label{eq::bilevel_quad}
\end{equation}
The quadratic term $\beta^{\shortminus 1}\left<\bar{y}-1,\bar{y}\right>$ is chosen so that it vanishes at any vertex of the simplex $\Delta_{W+1}$.
The following result shows that two perhaps more intuitive quadratic terms lead to the same optimum.
\begin{prop}\label{prop::quad_within_constant}
The two following penalizations are equivalent to the one in \eqref{eq::bilevel_quad}: \begin{enumerate}[label=(\roman*)]
\setlength\itemsep{-0.1\baselineskip}
  \item $ \frac{1}{\beta} \left\|\bar{y}_s - \frac{1}{W+1}\right\|_{W+1}^2$ (uniform law attractor),
  \item $\frac{1}{\beta} \left\|\bar{y}_s\right\|_{W+1}^2$.
\end{enumerate}
\end{prop}
\begin{proof}
$
\left\|\bar{y}_s-\alpha\right\|_{W+1}^2 - \left<\bar{y}_s-1, \bar{y}_{s}\right>_{W+1}
=\left(1-2\alpha\right)\left(\sum_{w=0}^W y_{sw}\right) + (W+1)\alpha^2  = (1-\alpha)^2+ W\alpha^2
$.\\
The two objective functions are equal up to a constant for valid lower responses, thus the argmins are the same.
\end{proof}
The first item suggests that our new penalization acts as an attractor to the uniform law whose intensity is inversely proportional to $\beta$. The bigger $\beta$ is, the more customers will uniformly spread their choices on all the possibilities. This asymptotic behavior is therefore identical to the one of logit regularization. In~\Cref{app::estimates}, we provide metric estimates -- along with illustrations -- in order to compare the logit and quadratic regularizations.
\shortciteA{Dempe_2001} have introduced such a quadratic regularization in order to avoid discontinuities that appears in the deterministic version $\eqref{eq::o_BPR_choices}$, and theoretically analyze the convergence of a bundle trust region algorithm specifically designed for this problem. Here, the second level is of a particular nature: we focus on lower problem defined on simplices, which allows us to interpret the customers' decision as a geometric object. In particular,  
for $W+1$ disutilities $V_{s0},\hdots, V_{sW}$, the follower response of a given segment $s$ can be written as 
\begin{equation}
\argmin_{\Delta_{W+1}}\left\{\sum_{w=0}^W V_{sw} y_{sw} + \frac{1}{\beta}y_{sw}^2\right\}
=\argmin_{\Delta_{W+1}}\left\|y_s - \left(-\frac{\beta}{2}V_s\right)\right\|_{W+1}
= \Proj_{\Delta_{W+1}}\left(-\frac{\beta}{2}V_s\right) \enspace .
\label{eq::proj_simplex}
\end{equation}
Here again, the disutility $V_{sw}$ of a segment $s$ stands for a certain $\theta_{sw}(x) -R_{sw}$ in the problem \eqref{eq::bilevel_quad}.
The response can be understood as a projection on the simplex of a specific vector whose intensity varies proportionally to $\beta$.

\begin{remark}The projection on a closed convex set is Lipschitz of constant one in the Euclidean norm, a fortiori, it is continuous. Therefore, the quadratic lower response $y^{quad}(x;\beta)$, solution of the lower problem in \eqref{eq::bilevel_quad}, is a continuous function of the price variables $x$. 
\end{remark}

\subsection{Lower Response and Leader's Profit}
In the logit model, the lower response of a segment $s$ is analytically known and is defined by the logit expression. To better understand the customer behavior, we aim to find an explicit calculation of the lower response for a segment $s$ that faces disutilities $V_{s0},\hdots,V_{sW}$. We assume that these disutilities are \emph{sorted in ascending order}. The lower response $y$  that satisfies \eqref{eq::proj_simplex} is the solution of the KKT conditions expressed as:
\begin{equation}
\begin{aligned}
&V_{sw} +\frac{2}{\beta} y_{sw} - \lambda_{sw} - \mu_s = 0, & w\in \{0\hdots W\}\\
&0 \leq y_{sw} \perp \lambda_{sw} \geq 0,& w\in \{0\hdots W\}\\
&y_s \in \Delta_{W+1},\,\lambda_s \in \bbR_{\geq 0}^{W+1},\,\mu_s \in \bbR
\end{aligned}
\label{eq::def_proba_quad}
\end{equation}
These conditions are necessary and sufficient because we study a convex minimization problem where the Slater's condition holds. In the sequel, we analyze the KKT system \eqref{eq::def_proba_quad} to characterize the customer's response.
\begin{lemma}[Monotonicity] If $y$ satisfies \eqref{eq::def_proba_quad}, the sequence $(y_{sw})_{w=0..W}$ is decreasing for disutilities sorted in ascending order.
\label{lemma::construction_sol_quad_1}
\end{lemma}
\begin{proof}
We consider $V_{sw_1} \leq V_{sw_2}$. If $y_{sw_2} = 0$, there is nothing to prove, the inequality $y_{sw_1} \geq y_{sw_2}$ is automatically satisfied.
If however $y_{sw_2} > 0$, $\lambda_{sw_2} = 0$ by complementarity, and therefore
$V_{sw_1} + \frac{2}{\beta}y_{sw_1} - \lambda_{sw_1} = V_{sw_2} + \frac{2}{\beta}y_{sw_2}$. 
Since $V_{sw_2} - V_{sw_1} \geq 0$ and $\lambda_{sw_1} \geq 0$, $\frac{2}{\beta}(y_{sw_1}-y_{sw_2}) \geq 0$. Therefore, in any case, $\forall w_1,w_2,\, V_{sw_1} \leq V_{sw_2} \Rightarrow y_{sw_1} \geq y_{sw_2}$.
\end{proof}
\begin{prop}[Lower response algorithm]\label{prop::construction_sol_quad}
For any segment $s$, let the sequence $\left(c_{sw}\right)_{w\in [W]}$ be
$$c_{sw} := \frac{1}{w}\left[\frac{2}{\beta}+\sum_{w'=0}^{w-1}V_{sw'}\right]$$
and let the index $\tau$ be defined as $\displaystyle \tau = \min\left\{w \in [W],\left\vert \,V_{sw} \geq c_{sw}\right.\right\} $.
Then, the sequence $(c_{sw})$ verifies the following property:
\begin{equation}
V_{sw} < c_{s\tau} \text{ for } w < \tau\,;\;V_{sw} \geq c_{s\tau} \text{ for } w \geq \tau\label{eq::pre_cell}\end{equation}
Moreover, the solution $(y_s,\lambda_s,\mu_s)$ of \eqref{eq::def_proba_quad} can be expressed as follows 
\begin{enumerate}[label=(\roman*)]
\setlength\itemsep{-0.1\baselineskip}
  \setlength\itemsep{-0.1\baselineskip}
  \item $y_{sw} = \frac{\beta}{2}\left[c_{s\tau} - V_{sw}\right]$ for $w < \tau$ \,;\;  $y_{sw} = 0$ for $w\geq\tau$,
  \item $\lambda_{sw} = 0$ for $w < \tau$\,;\; $\lambda_{s\tau} \,= V_{s\tau} - c_{s\tau}$\,;\;$\lambda_{sw} = \lambda_{s,w-1} + V_{sw} - V_{s,w-1}$ for $w>\tau$,
  \item $\mu_s = c_{s\tau}$.
\end{enumerate}
The index $\tau$ is therefore the index from which the probability $y$ becomes zero.
\end{prop}
\begin{proof}
The first property on $(c_{sw})$ comes with the ascending sort of $V_s$ and the definition of $\tau$: $V_{s\tau} \geq c_{s\tau}$ and therefore $V_{sw} \geq c_{s\tau}$ for $w\geq \tau$. Besides, by minimality of $\tau$, $V_{s,\tau-1} < c_{s,\tau-1}$. Using the definition of $(c_{sw})$, for all $w<\tau$, $V_{sw} \leq V_{s,\tau-1} = c_{s\tau} - \frac{\tau-1}{\tau}(c_{s,\tau-1}-V_{s,\tau-1}) < c_{s\tau}$.

Concerning the second part of the proposition, one can first remark that solution of \eqref{eq::def_proba_quad} is unique since it is a projection on the simplex, see \eqref{eq::proj_simplex}.
The procedure returns a certain $(y_s,\lambda_s,\mu_s)$ which is feasible for \eqref{eq::def_proba_quad}:
by construction, $y$ is nonnegative, $\sum_{w=0}^{\tau} y_{sw} = 1$ and the complementarity constraints are satisfied. As the disutilities are sorted, $\lambda_{sw}\geq\lambda_{s\tau}\geq 0$ for any $w\geq \tau$. The solution we obtain is therefore the unique solution of~\eqref{eq::def_proba_quad}.
\end{proof}

From the explicit calculation of the lower response, one can observe the following property
\begin{corol}[Soft threshold]\label{corol::threshold}
If $y_s$ satisfies \eqref{eq::def_proba_quad}, the first disutility is chosen with probability $1$ if and only if the difference between any other disutility and the one chosen is higher than $2/\beta$ i.e., 
\begin{equation}
y_{s0} = 1 \text{ and } \forall w > 0,\, y_{sw} = 0 \,\iff\, \forall w>0,\, V_{sw} \geq V_{s0} + \frac{2}{\beta} \enspace .\label{e-threshold}\end{equation}
\end{corol}
\begin{proof}
From the last proposition, the condition $\forall w > 0,\,V_{sw} \geq V_{s0} + \frac{2}{\beta}$ is equivalent to $V_{s1} \geq V_{s0} + \frac{2}{\beta}$ which means that $\tau=1$.
\end{proof}

Coming back to problem \eqref{eq::bilevel_quad}, we summarize the properties of the lower response in the following corollary
\begin{corol}[Lower response of \eqref{eq::bilevel_quad}]\label{corol::quad_lower_response}
For a price strategy $x$ and a given $\beta$, the quadratic lower response $\left(y^{quad}_{sw}(x;\beta)\right)_{w=0\hdots W}$ for a segment $s$ can be computed by the following algorithm:
\begin{enumerate}
\setlength\itemsep{-0.1\baselineskip}
  \item Compute $V_{sw}(x) := \theta_{sw}(x) - R_{sw}$ for all $w\in [W]$ and $V_{s0} = 0$,
  \item Reindex the disutilities so that they are sorted in the ascending order,
  \item Calculate the solution $y$ defined in~\Cref{prop::construction_sol_quad},
  \item The value $y^{quad}_{sw}(x;\beta)$ is the component of $y$ that corresponds to the disutility $V_s$ initially indexed by $w$.
\end{enumerate}
\end{corol}

As pointed out in equation \eqref{eq::proj_simplex}, the lower response can be viewed as a projection on the simplex. Fives algorithms to compute the projection are provided in \shortciteA{Condat_2016}. The first one, applied to the projection $\Proj_{\Delta_{W+1}}\left(-\frac{\beta}{2}V_s\right)$, allows us to recover the response found in \Cref{prop::construction_sol_quad}. Other algorithms are faster but do not contain such a clear interpretation that customers select disutilities with the lowest values.

The threshold that appears in \cref{corol::threshold} suggests a link with the work of~\shortciteA{Shioda_2011}, where they constrain the price strategy to ensure a minimal gap between the lowest disutility and the other ones. Our result shows a soft threshold effect at a finite rationality ($\beta < \infty$). We allow the variables $y_{sw}$ to be fractional values, but they will concentrate on a unique contract per segment if the disutilities are sufficiently separated. This effect only occurs asymptotically ($\beta = \infty$) in the logit model. 
\Cref{corol::threshold} also has an intuitive economic interpretation. In fact, one can link the estimation of the regularization intensity $\beta$ with the minimal gap (in \euro) above which the decision coincides with the best deterministic one (probability one to choose the offer giving the highest utility). For example, the $2/\beta$ threshold in~\eqref{e-threshold} reveals that a value $\beta = 0.2$ then corresponds to the minimal difference of $10$\euro\; to recover a binary decision.

For a given price strategy $x$, the leader profit function $\pi^{quad}(x;\beta)$ is then defined as 
\begin{equation}
\pi^{quad}(x;\beta):=\sum_{s\in [S]} \rho_s \sum_{w \in [W]}(\theta_{sw}(x)-C_{sw})y^{quad}_{sw}(x;\beta)
\label{eq::pi_quad}
\end{equation}
where $y^{quad}(x;\beta)$ is defined as explained in \cref{corol::quad_lower_response}.
The problem $(q \trt BP)$ is therefore the maximization of the function $\pi^{quad}$ over $X$.


\subsection{Price complex and convergence to the deterministic model}
\shortciteA{Baldwin_2019} have introduced a geometric approach to analyze the response of agents to prices, in a discrete choice model. They showed that the deterministic response is governed by a polyhedral complex: all prices in a given cell yield the same response.
Here, we generalize this approach to continuous responses, since in our regularized
model, responses do not concentrate anymore on a single contract. However,
the closed-form formula we found for the lower response highlights the sparsity in terms of customers choices. In fact, in a feasible solution, only few contracts have positive probabilities to be chosen by a segment $s$ (we call them \emph{active contracts}). 
Now, all the prices in a given cell yield (different) responses encoded with the same ``sparcity pattern" i.e., the responses share the same set of active contracts.

\begin{defin}\begin{enumerate}
\setlength\itemsep{-0.1\baselineskip}
\item A matrix $A \in \{0,1\}^{S \times (W+1)}$ is called a \emph{pattern}. We denote by $\vert A_s \vert$ the number of positive coefficients in row $s$, and $\vert A \vert = \sum_{s\in[S]}\vert A_s\vert$ the total number of positive coefficients of $A$.
\item We denote by $X(A;\beta)$ the price strategies that have an active-contracts set corresponding to the pattern $A$, i.e.,
$$X(A;\beta) := \left\{x \in X\,\vert \, \mathds{1}_{(y^{quad}_{sw}(x;\beta)>0)} = A_{sw},\,\forall s,w\right\}\enspace .$$
Thus, the set of active contracts stays unchanged on the set $X(A;\beta)$. We call this price region a \emph{unique pattern region} (UPR).
\end{enumerate}
\end{defin}
The UPRs are not closed since we look at the prices that give a positive probability. Thus, we define $\overline{X}(A;\beta)$ to be the closure of the UPR $X(A;\beta)$. 
\begin{defin}\label{defin::price_cells}
\begin{enumerate}
\setlength\itemsep{-0.1\baselineskip}
\item A pattern $A$ is said to be \emph{feasible} if $\overline{X}(A;\beta)$ is non-empty, and $\mathcal{A^\beta} \subseteq \{0,1\}^{S\times (W+1)}$ is then the set of feasible patterns.
\item A \emph{pure} pattern $A$ is a pattern containing only pure strategies i.e., each segment has a unique active contract $(\vert A\vert = S)$. The other patterns are called \emph{mixed} patterns $(\vert A\vert > S)$. 
\item A \emph{price complex cell} is a non-empty set $P \subseteq X$ such that there exist $A^1,\hdots,A^k \in \mathcal{A}^\beta$, with $k\geq 1$, satisfying $P=\bigcap\limits_{1 \leq i \leq k} \overline{X}(A^i;\beta)$.
\item The \emph{price complex} is the collection of all price complex cells.
\end{enumerate}
\end{defin}

\begin{prop}[Characterization of the price complex cells]\label{prop::polyhedral_complex}
For any pattern $A\in \mathcal{A}^\beta$ and any $\beta > 0$, the UPR $X(A;\beta)$ is defined as $X(A;\beta) = \overline{X}^0(A;\beta) \cap X^1(A;\beta)$ where 
\begin{subequations}\label{eq::poly_cell}
\begin{align}
\overline{X}^0(A;\beta) &= \left\{x\in X \left\vert \begin{aligned}
  &\forall s,w, \text{ if } A_{sw}=0,\\
  &\quad \vert A_s\vert  V_{sw}(x) \geq 2\beta^{\shortminus 1}+\sum_{w'\,\vert \,A_{sw'} = 1} V_{sw'}(x)\end{aligned}
\right.\right\},\label{eq::poly_cell_0}\\
X^1(A;\beta) &=\left\{x\in X \left\vert \begin{aligned}
  &\forall s,w, \text{ if } A_{sw}=1,\\
  &\quad \vert A_s\vert  V_{sw}(x) < 2\beta^{\shortminus 1}+\sum_{w'\,\vert \,A_{sw'} = 1} V_{sw'}(x)\end{aligned}
\right.\right\}.\label{eq::poly_cell_1}
\end{align}
\end{subequations}
where $\vert A_s\vert $ corresponds to the number of active contracts for $s$ and $V_{sw}(x)$ is defined as in \Cref{corol::quad_lower_response}. 
As a consequence, $\overline{X}(A;\beta) = \overline{X}^0(A;\beta)\cap\overline{X}^1(A;\beta)$ where $\overline{X}^1(A;\beta):=\cl\left(X^1(A;\beta)\right)$, obtained by weakening the inequalities~\eqref{eq::poly_cell_1}.
\end{prop}
\begin{proof}
Given a pattern $A\in \mathcal{A}^\beta$ and a $\beta > 0$, we can assume w.l.o.g. that for any segment $s$ the disutilites are sorted in ascending order so that the active contracts are the first $\vert A_s\vert $ ones. 
First, we consider a price strategy $x \in X(A;\beta)$. Using the notation of \Cref{prop::construction_sol_quad}, $\tau = \vert A_s\vert $ and equation \eqref{eq::pre_cell} gives us exactly that $x\in \overline{X}^0(A;\beta)\cap X^1(A;\beta)$.
Reciprocally, we suppose that $x\in\overline{X}^0(A;\beta)\cap X^1(A;\beta)$\eqref{eq::poly_cell} is satisfied i.e.,
$
V_{sw} < c_{s,\vert A_s\vert }$  for $w < \vert A_s\vert$ and $V_{sw} \geq c_{s,\vert A_s\vert }$ for $w \geq \vert A_s\vert$.
Then, $\tau = \vert A_s\vert $ and $x \in X(A;\beta)$.
\end{proof}

\begin{theorem}\label{prop::poly_complex}
The collection of price complex cells constitutes a $\vert X \vert$-dimensional polyhedral complex, and the $\vert X \vert$-cells are closures of UPRs i.e., $\overline{X}(A;\beta)$ for some pattern $A\in \mathcal{A}^\beta$.
\end{theorem}
\begin{proof}
It is clear that the collection of price complex cells covers the space $X$. Besides, from the definition of a cell, the intersection of two cells $P$ and $P'$ is again a price complex cell or is empty. Finally, \Cref{prop::polyhedral_complex} gives us a characterization of the cells with linear inequalities, therefore the intersection of $P$ with another $P'$ is then characterized by the same inequalities as $P$ but with some of them saturated. Hence, the intersection is a common face of $P$ and $P'$.
\end{proof}

We now study the asymptotic behavior of the price complex ($\beta\to\infty$) and show how it embeds in the deterministic complex introduced in~\shortciteA{Baldwin_2019}.  To this end, we first denote by $\overline{X}(A;\infty)$ the polytope defined by the same inequalities as in $\overline{X}(A;\beta)$ setting $\beta^{\shortminus 1}=0$ (idem for $\overline{X}^0$ and $\overline{X}^1$), and by $\mathcal{A}^\infty$ the set of patterns inducing a non-empty $\overline{X}(A;\infty)$.
We next make use of the notion of Painlevé-Kuratowski limits of sets. We refer to~\cite[Chapter 4]{RW_2009} for background on this notion, including the definition and properties of upper and lower limits.

\begin{prop}[Convergence] \label{prop::convergence_cell}For any pattern $A$, $\limsup_{\beta\to\infty} \overline{X}(A;\beta) \subseteq \overline{X}(A;\infty)\enspace .$
Moreover, if $\Int\left(\overline{X}(A;\infty)\right)\neq \varnothing$,
$$\overline{X}(A;\beta) \xrightarrow[\beta]{} \overline{X}(A;\infty)\enspace.$$
\end{prop}
\begin{proof} See~\Cref{app::proof_convergence_cell}.
\end{proof}
\begin{lemma}\label{lemma::infinite_charact}
For any pattern $A\in\mathcal{A}^\infty$, the asymptotic cell $\overline{X}(A;\infty)$ can be equivalently defined by the following system 
\begin{equation}\label{eq::infinite_charact}
\begin{aligned}
\forall s, w,w',\,&\text{if } A_{sw} = A_{sw'} = 1, & V_{sw}(x)=V_{sw'}(x),\\
&\text{if } A_{sw} = 1 \text{ and } A_{sw'} = 0, & V_{sw'}(x)\geq V_{sw}(x).
\end{aligned}
\end{equation}
\end{lemma}
\begin{proof}
We first define the mean active disutility for a segment $s$ as $\tilde{V}_s = \frac{1}{\vert A_s \vert}\sum_{w'\,\vert\, A_{sw'}=1} V_{sw'}$.
Then, we know by \eqref{eq::poly_cell_1} that for any active contract $w$,
$V_{sw} - \tilde{V}_s \leq \frac{2}{\beta}$. Denoting by $V^+$ and $V^-$ the extreme disutilities of active contracts, we obtain $0\leq V^+-V^-\leq \frac{2}{\beta}$. At the limit, active contracts share a same disutility, equal to $\tilde{V}_s$.
Besides, we also know from \eqref{eq::poly_cell_0} that for any inactive contract $w$, $V_{sw}\geq \frac{2}{\beta} + \tilde{V}_s$. At the limit, any inactive contract has disutilities greater than the active contracts.
\end{proof}
\begin{lemma}For any mixed pattern $A$, there exist $k>1$ pure patterns $A^1,\hdots,A^k$ such that 
$$\overline{X}(A;\infty) = \bigcap_{1\leq i\leq k}\overline{X}(A^i;\infty)\enspace.$$
\end{lemma}
\begin{proof}
Suppose that for a given segment $s$, $\vert A_s\vert = k$, then we can construct patterns $A^1,\hdots,A^k$ such that $A^i$ is a copy of $A$ where the row $s$ is replaced by $1$ on the $i$th active contract, and $0$ everywhere else. From the characterization~\eqref{eq::infinite_charact}, we obtain that $\overline{X}(A;\infty) = \bigcap_{1\leq i\leq k}\overline{X}(A^i;\infty)$. Each pattern $A^i$ has pure strategy for segment $s$. If there still exist mixed strategies for other segment, we can start again the transformation until all the patterns are pure.
\end{proof}
At the limit $\beta=\infty$, each mixed pattern is a face of some pure patterns. The pure patterns are therefore sufficient to describe any cell.
\begin{theorem}[Asymptotic cells and UPRs]\label{theorem::asymptotic_cells}
Let $A^1,\hdots,A^k$ be $k$ pure patterns, then
$$x\in P= \bigcap\limits_{1 \leq i \leq k} \overline{X}(A^i;\infty) \iff \{A^1,\hdots,A^k\} \subseteq \Psi(x)\enspace.$$
where $\Psi(x)$ is the set of optimistic best responses, see~\eqref{eq::o_lower_problem}.
Moreover, for any pure pattern $A$, $$\Int\left(\overline{X}(A;\infty)\right) = \left\{x\in X:\,\{A\} = \Psi(x)\right\}\enspace.$$
\end{theorem}
\begin{proof}
The equivalence is a direct consequence of the~\Cref{lemma::infinite_charact}. The equality also arises from this lemma: the set $\{x\in X:\,\{A\}=\Psi(x)\}$ is characterized by \eqref{eq::infinite_charact} with strict inequalities.
\end{proof}

\Cref{theorem::asymptotic_cells} establishes a link with the approach of~\shortciteA{Baldwin_2019}: we generalize the price complex to relaxed choices and the definition we introduce in \Cref{defin::price_cells} is equivalent to their definition in the specific case $\beta=\infty$. Moreover, Baldwin and Klemperer define \emph{unique demand region} (UDR) where the set $\Psi(x)$ has a unique element, and~\Cref{theorem::asymptotic_cells} proves that any pure UPR converges to the corresponding UDR.
\begin{figure}[!ht]
  \centering
  \includegraphics[width=0.6\linewidth, clip = true, trim = 1.cm 0cm 1.5cm 1.5cm]{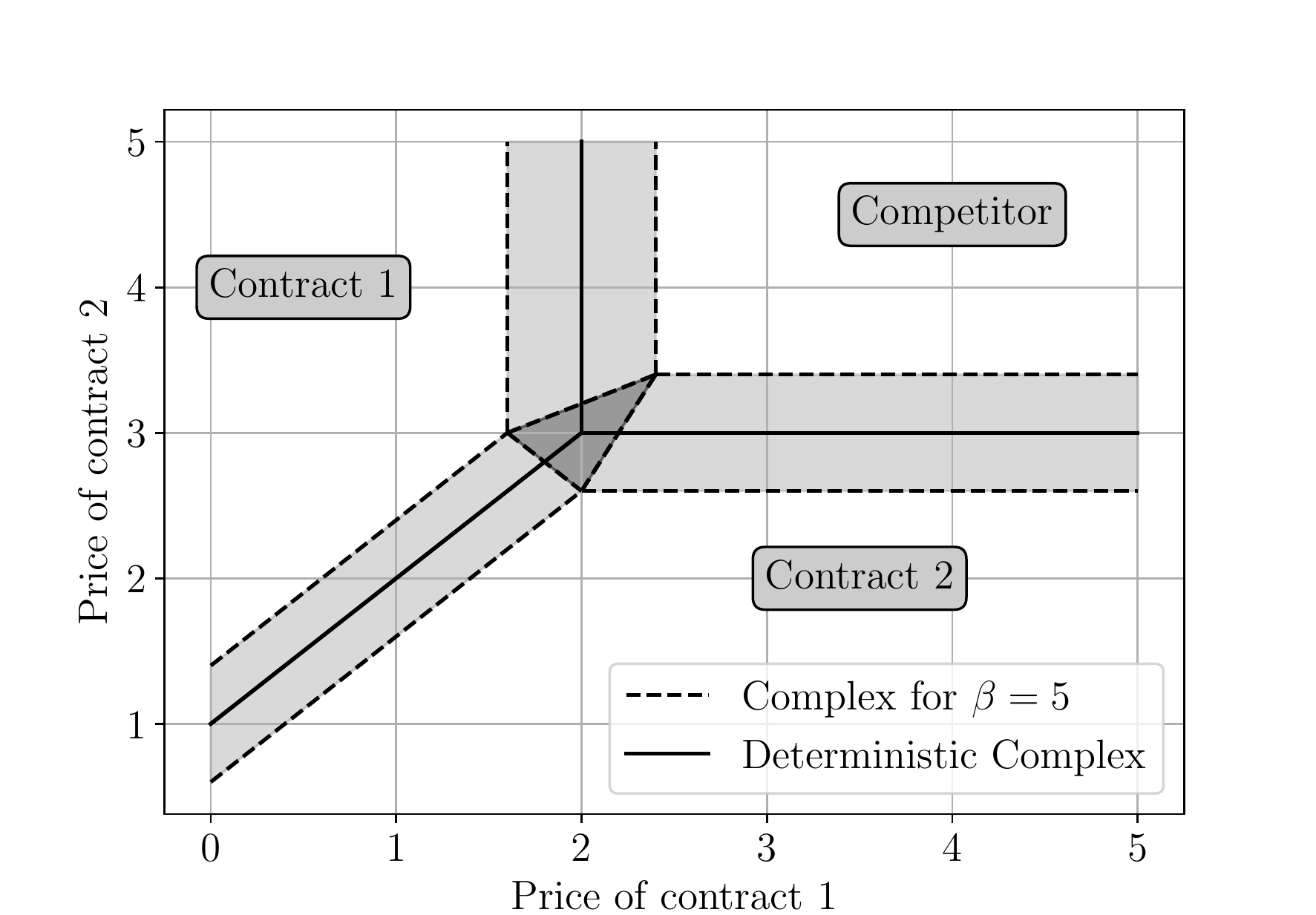}
  \caption{Price complex in a simple case.} 
\medskip 
\small
For $\beta=\infty$ (deterministic case, solid line), the three cells correspond to the the choice of a unique contract (rectangles indicate the choice). For $\beta < \infty$, each line ``splits" to create intermediate cells (mixed strategies). Pure strategies correspond to white zones, strategies mixing two contracts correspond to light gray zones and the strategy mixing all contracts corresponds to the dark gray zone.
  \label{fig::poly_complex}
\end{figure}
To illustrate~\Cref{prop::polyhedral_complex}, \Cref{fig::poly_complex} shows the complex cells for a single customer making a choice among two contracts from the company and one from a competitor. The deterministic complex was depicted in~\shortcite[Figure 1]{Baldwin_2019} or in~\shortciteA{Eytard_2018} for bilevel models, and \Cref{fig::poly_complex} illustrates the generalization of the price complex to relaxed choices: note that new types of full-dimensional cells, representing choices
concentrated on several contracts, appear. 

The logit profit function has no good convexity properties in our context of a heterogeneous population. Thanks to the properties of the lower response and the notion of polyhedral complex, we can prove that its quadratic analog is more structured:
\begin{lemma}\label{lemma::convexe_J}
For $K\geq N$, the function $J: x\in \bbR^N \mapsto \sum_{i=1}^N x_i^2 - \frac{1}{K}\left(\sum_{i=1}^N x_i\right)^2$ is convex.
\end{lemma}
\begin{proof}
The Hessian $H$ of the function $J$ is $H_{ij} = -2/K$ for $i\neq j$ and $H_{ii} = 2 - 2/K$.
Using the Gershgorin circle theorem, any eigen value $\lambda_i$ has to verify 
$\vert \lambda_i - (2 - 2/K)\vert \leq \sum_{j\neq i} 2/K$.
Therefore, $\lambda_i \geq 2  - 2N/K$ and we deduce that all eigen values of $H$ are nonnegative.
\end{proof}
\begin{theorem}[Profit decomposition]\label{prop::quad_profit_piecewise}
The quadratic leader profit function $\pi^{quad}(x;\beta)$ is continuous. Moreover, the problem \eqref{eq::bilevel_quad} is equivalent to the following problem
\begin{equation}
\max_{A \in \mathcal{A}^\beta} \left\{ \varphi(A;\beta) := \max_{x\in\overline{X}(A;\beta)} \pi^{quad}(x;\beta)\right\}
\label{eq::quad_by_cells}
\end{equation} 
where $\pi^{quad}(x;\beta)$ is concave on each price complex cell $\overline{X}(A;\beta)$, defined in~\Cref{prop::polyhedral_complex}.
\end{theorem}
\begin{proof}
The continuity of the lower response suffices to ensure the continuity of $\pi^{quad}$. The difficulty lies in the concave foundation.
Because the profit function is a sum over the segments, we may assume that there is only one segment $s$.
Let us consider a feasible pattern $A \in \mathcal{A}^\beta$. On the cell $\overline{X}(A;\beta)$ associated with this pattern, the profit function is expressed as
$$J^A_s(x) := \sum_{w \in [W] \,\vert \, A_{sw}=1}(\theta_{sw}(x)-C_{sw})y^{quad}_{sw}(x;\beta)\enspace.$$
To keep compact notation, we define $\calW^A_s := \{w \in [W] \,\vert \, A_{sw}=1\}$, and $V_{sw} := \theta_{sw}(x) - R_{sw}$ for $w \in [W]$ and $V_{s0}=0$.
Using \cref{corol::quad_lower_response}, we can rewrite $J^A$ as 
$$
\begin{aligned}
J^A_s(x) &= \frac{\beta}{2}\sum_{w \in \calW^A_s}(V_{sw}+R_{sw}-C_{sw})(c_{s,\vert A_s\vert } - V_{sw})\\
&= \frac{\beta}{2}\sum_{w \in \calW^A_s}(R_{sw}-C_{sw})(c_{s,\vert A_s\vert } - V_{sw})
- \frac{\beta}{2}\left[\sum_{w \in \calW^A_s}V_{sw}^2 - c_{s,\vert A_s\vert }\sum_{w \in \calW^A_s}V_{sw}\right]\\
&=L - \frac{\beta}{2}\left[\sum_{w \in \calW^A_s}V_{sw}^2 
- \frac{1}{\vert A_s\vert }\left(\sum_{w \in \calW^A_s}V_{sw}\right)^2\right]
\end{aligned}
$$
where $L=\frac{1}{\vert A_s\vert}\sum_{w \in \calW^A_s}V_{sw} + \frac{\beta}{2}\sum_{w \in \calW^A_s}(R_{sw}-C_{sw})(c_{s,\vert A_s\vert } - V_{sw})$ denotes the linear part.
The set $\calW^A_s$ has a cardinality of $\vert A_s\vert $ or $\vert A_s\vert -1$ depending on if the no-purchase option appears in the first $\vert A_s\vert $ disutilities. Therefore,
by~\Cref{lemma::convexe_J},
$J^A_s$ is concave in $V_s$, and thus  is concave in $x$ since the functions $\theta$ are linear. 
Finally, exploring $\overline{X}(A;\beta),A\in \mathcal{A}^\beta$ is sufficient to cover the whole space $X$.
\end{proof}

\Cref{prop::quad_profit_piecewise} paves the way to enumerative scheme resolutions: it shows that the problem can be polynomially solved on each cells of the polyhedral complex, and if all the cells are explored it gives a global optimum. Nonetheless, it could be very cumbersome (especially for low $\beta$ values). 


\subsection{QPCC Reformulation}
As in the deterministic case, the model can be recast into a single-level program with complementarity constraints using the KKT conditions. Moreover, we are able to replace the bilinear terms using manipulations on the constraints:
\begin{theorem}\label{prop::QPCC_formulation}The problem \eqref{eq::bilevel_quad} is equivalent to the following concave QPCC problem
\begin{equation}\label{eq::beta_QPCC}
\tag{$q\beta \trt QPCC$}
\begin{aligned}
\max_{x\in X,\mu\in\bbR^S,\bar{y}} &\: \sum_{s\in [S]} \rho_s \mu_s + \rho_s\left<R_s- C_s, y_s\right>_W - 2\beta^{\shortminus 1}\rho_s\left\|\bar{y}_{s}\right\|_{W+1}^2\\
~
\st &\; 0 \leq y_{sw} \perp \theta_{sw}(x) - R_{sw} + 2\beta^{\shortminus 1}y_{sw} - \mu_s \geq 0,\,\forall s,w \\
&\; 0 \leq y_{s0} \perp 2\beta^{\shortminus 1}\bar{y}_{s} - \mu_s \geq 0,\,\forall s\\
&\; \bar{y}_s \in \Delta_{W+1},\,\forall s
\end{aligned}
\end{equation}
\end{theorem}
\begin{proof}
The KKT optimality condition have been detailed in \eqref{eq::def_proba_quad}. One can remark that the variable $\lambda$ can be removed to obtain the KKT system of~\eqref{eq::beta_QPCC}.
We then reformulate the objective by using the constraints: for a given $s\in [S]$,
$$\begin{aligned}
\left<\theta_s(x), y_s\right>_W &= \left<\mu_s e_{W} + R_s, y_s\right>_W - 2\beta^{\shortminus 1}\left\|y_{s}\right\|_W^2\\
&= \mu_s -\mu_s y_{s0} +\left<R_s, y_s\right>_W - 2\beta^{\shortminus 1}\left\|y_{s}\right\|_{W}^2\enspace .
\end{aligned}$$
Finally, the objective in \eqref{eq::beta_QPCC} is obtained using the complementarity constraint on the no-purchase option:  $\mu_s y_{s0} =2\beta^{\shortminus 1}y_{s0}^2$.
\end{proof}

As in the deterministic case, we can replace the complementarity constraints in \eqref{eq::bilevel_quad} by Big-$M$ constraints to obtain a mixed-integer quadratic problem (MIQP). However, the introduction of binary variables is unavoidable since the existence of a solution with integer lower response is no longer true:

\begin{equation}
\begin{aligned}
\max_{x\in X,\mu\in\bbR^S,\bar{y},z} &\quad \sum_{s\in [S]} \rho_s \mu_s + \rho_s\left<R_s- C_s, y_s\right>_W - 2\beta^{\shortminus 1}\rho_s\left\|\bar{y}_{s}\right\|_{W+1}^2\\
~
\st \quad &\quad 0 \leq \theta_{sw}(x) - R_{sw} + 2\beta^{\shortminus 1}y_{sw} - \mu_s \leq M_{sw}(1-z_{sw}),\,\forall s, w \\
&\quad 0 \leq 2\beta^{\shortminus 1}\bar{y}_{s} - \mu_s \leq M_{s0}(1-z_{s0}),\,\forall s\\
&\quad y \leq z\\
&\quad \bar{y}_s \in \Delta_{W+1},\,z\in \{0,1\}^{W+1},\,\forall s
\end{aligned}
\label{eq::beta_MIQP}
\end{equation}

QPCC problems have been recently studied, using conic relaxations -- \shortciteA{Deng_2017,Zhou_2019} -- or logical Benders -- \shortciteA{Bai_2013,Moroni_2020}. In the latter, they introduce the notion of \emph{complementarity piece} defined by a valuation of the binary vector $z$. The complementarity pieces of \eqref{eq::beta_QPCC} coincide with the cells $\overline{X}(A;\beta)$ of the price complex \eqref{eq::poly_cell}: admissible valuations of $z$ define feasible patterns, and vice versa.

\subsection{Comparison with logit model}\label{sec::comp_logit}
Quadratic and logit regularizations share a parameter $\beta$, interpreted as a rationality parameter. It will be convenient to replace the regularization parameter $\beta$ in the quadratic model by $\beta'=\beta e/4$, leaving the value $\beta$ in the logit model. In fact, the minimum of $\frac{1}{\beta}y(y-1)$ is $-\frac{1}{4\beta}$ whereas the minimum of  $\frac{1}{\beta}y\log(y)$ is $-\frac{1}{e\beta}$, and so this choice of $\beta$ equalizes the minimal intensity of the regularization term.
To have a better intuition on the differences and similarities between the logit and quadratic regularization, we study a simple case where there is one single-attribute contract and five customers.  We provide in  \cref{fig::comparison_regularization} the leader profit as a function of the contract price for multiple configurations (the optimistic version, the quadratic version and the logit version for two values of $\beta$).
\begin{figure}[!ht]
  \centering
  \includegraphics[width=.75\linewidth, clip = true, trim = 1.5cm 0.2cm 2.cm 1.7cm]{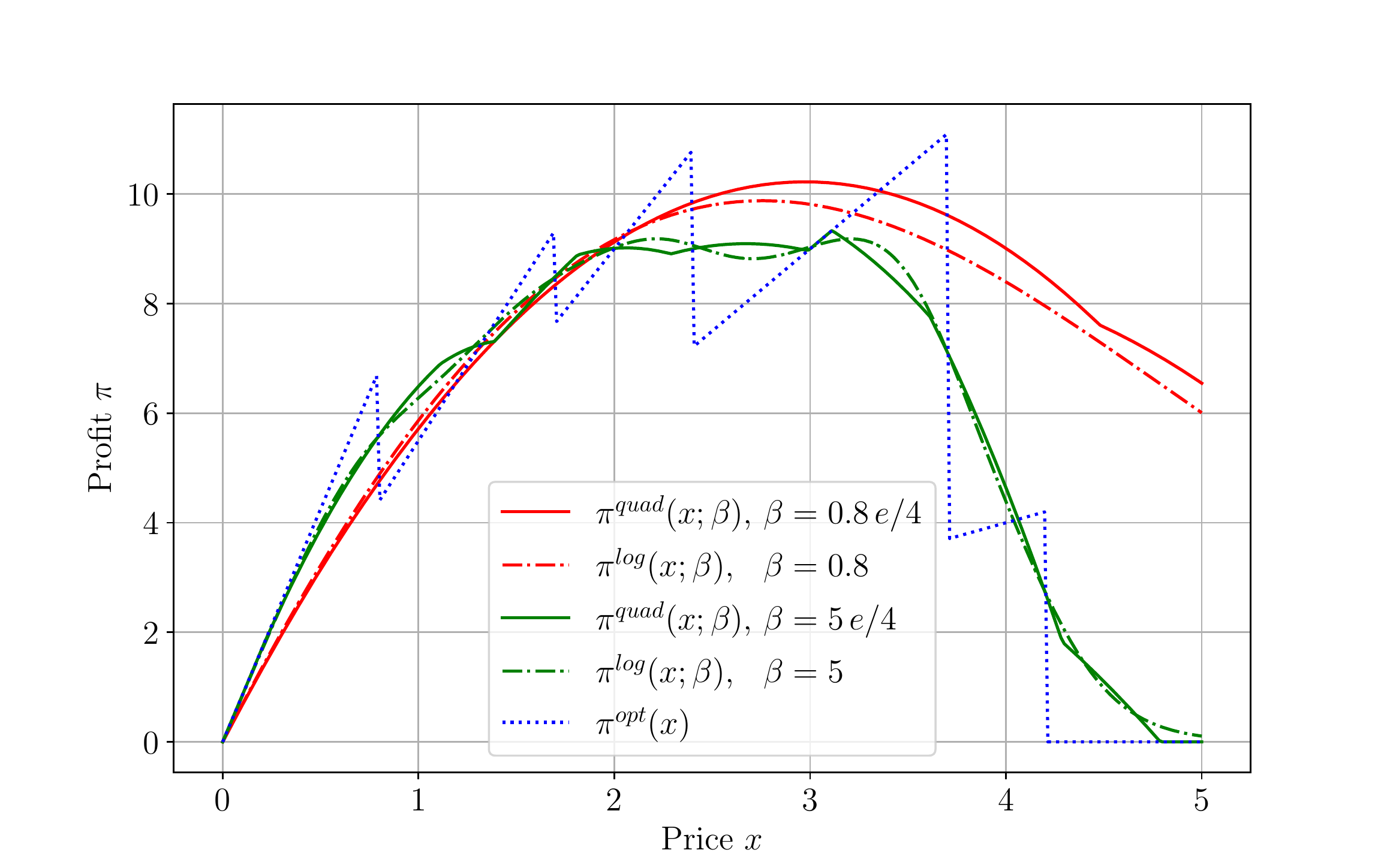}
  \caption{Comparison of profit functions $\pi^{opt}$, $\pi^{log}$ and $\pi^{quad}$}
  \label{fig::comparison_regularization}
\end{figure}
The behavior of the deterministic and logit profit have already been compared in another context in \shortciteA{Gilbert_2015}. We now include the quadratic model in this comparison. The following properties of profit functions can be identified: 
\begin{itemize}
  \setlength\itemsep{.1em}
\item The deterministic profit is piecewise linear  but contains discontinuities that arise when two contracts share the same minimal disutilities (here between the only contract and the no-purchase option). 
The optimal profit is always attained at such a frontier price, leading to an instability: for this specific case, the optimal deterministic profit is higher than $11$ and is achieved for $x=3.7$. Nevertheless, a price of $x=3.71$ induces a profit lower than 4. 
\item The logit regularization smooths the deterministic profit function while maintaining its global shape for $\beta$ large enough. Nonetheless, the function is non-convex and we can observe for $\beta = 0.8$ two local maxima.
\item The quadratic regularization and its logit analog share the same behavior: in fact, the shape is very similar for both values of $\beta$. The difference lies in the structure of the quadratic model: the profit function is piecewise concave, see~\Cref{prop::quad_profit_piecewise}.
\end{itemize}
\section{Local Search by Pivoting on the Price Complex}\label{sec-localsearch}
In the previous sections, we established geometrical properties of the quadratic regularization. In particular, \Cref{prop::QPCC_formulation} provides a direct formulation which allows us to find a global optimum via MIQP techniques. However, such methods are workable only up to a limited instance size, above which a good optimality gap cannot be obtained in reasonable time. Therefore, it is of interest to develop a local search method taking advantage of the structure highlighted in \Cref{prop::quad_profit_piecewise}: finding the optimal solution is no more than finding the cell of the polyhedral complex containing this solution.
Indeed, computing the optimum on a given cell reduces to a (simple) quadratic program.
Given a cell, a neighbor cell can be obtained by reversing one of the inequalities~\eqref{eq::poly_cell}, however, computing all the neighbor cells
is computationally expensive (there is a large number
of inequalities~\eqref{eq::poly_cell} and moreover some of them are redundant).
Hence, we introduce a narrow neighborhood which selects specific neighbors obtained by reversing the inequalities associated to contracts near the active/inactive frontier, as these yields good candidates in the search for better solutions, see~\Cref{algo::explore_good_neighbors}.

\begin{algorithm}[!ht]
\caption{\texttt{exploreGoodNeighbors}}\small
\begin{algorithmic}[1]
\Require $A,x_A,\varphi_A =\pi^{quad}(x_A;\beta)$ \Comment{$x_A$ optimum on the initial pattern $A$}
\State $A^*,x^*,\varphi^* \leftarrow A,x_A,\varphi_A$ \Comment{$A^*$ will be the best neighboring pattern} 
\For{$s = 1 \hdots S$}
\State $A^- \leftarrow A$, $A^+ \leftarrow A$
\State $w^-\leftarrow \max_{0\leq w\leq W} \left\{V_{sw} \,\vert\,A_{sw} = 1\right\}$ \Comment{worst active contract}
\State $w^+\leftarrow \min_{0\leq w\leq W} \left\{V_{sw} \,\vert\,A_{sw} = 0\right\}$ \Comment{best nonactive contract}
\State $A^-_{s,w^-},A^+_{s,w^+} \leftarrow 0,1$ \Comment{new patterns}
\For{$\dagger \in \{-,+\}$}
\State $x^\dagger,\varphi^\dagger \leftarrow$ solution of~$\max_{x \in \overline{X}(A^\dagger;\beta)} \pi^{quad}(x;\beta)$
\If{$\varphi^\dagger \geq \varphi^*$} \State $A^*,x^*,\varphi^* \leftarrow A^\dagger, x^\dagger,\varphi^\dagger$
 \Comment{update best pattern} \EndIf 
\EndFor
\EndFor
\State\Return $A^*$, $x^*$, $\varphi^*$
\end{algorithmic}
\label{algo::explore_good_neighbors}
\end{algorithm}

More precisely, for a given feasible pattern $A$ and for any segment $s$, we select two inequalities on which we will pivot:
\begin{enumerate}[label=(\roman*)]
\setlength\itemsep{.1em}
\item the inequality~\eqref{eq::poly_cell_1} of index $(s,w^-)$ where $w^-$ is the active contract with the greatest disutility for $s$ (i.e., with the lowest positive probability $y_{sw}$),
\item and the inequality~\eqref{eq::poly_cell_0} of index $(s,w^+)$ where $w^+$ is the non-active contract with the lowest disutility.
\end{enumerate}
By pivoting, we mean that, starting from this feasible pattern
$A$, we consider a new cell,
in which all the inequalities in~\eqref{eq::poly_cell_0}
and~\eqref{eq::poly_cell_1} stay unchanged, except the two ones of indices
$(s,w^-)$ and $(s,w^+)$ that are reversed. This leads
to a new pattern. 

Although this strategy does not explore the whole neighborhood (it only changes $2S$ inequalities among $WS$), pivoting on the selected inequalities is likely to produce relevant new cells. 
We will consider two methods for exploring these cells:
\begin{enumerate}
  \setlength\itemsep{.1em}
  \item computing $\varphi(\,\cdot\,;\beta)$ for each of the $2S$ neighboring patterns by solving $2S$ quadratic programs, see \eqref{eq::quad_by_cells}, and returning the best pattern $A'$ with its value $\varphi(A';\beta)$ (it could be the initial pattern if no improvement was made),
  \item or solving the MIQP \eqref{eq::beta_MIQP} where the only unfixed binary variables $z$ are the $2S$ variables indexed by the selected inequalities (the other variables $z$ are equal to the current pattern values) and returning the pattern $A'$ obtained by the solver with its value $\varphi(A';\beta)$.
\end{enumerate}
The second option is computationally more expensive (as it relies on a MIQP)
but it explores a wider neighborhood, since several of the $S$-groups of $2$
inequalities can be reversed in a single step.

Iterating the procedure~\texttt{exploreGoodNeighbors} (\Cref{algo::explore_good_neighbors}) produces a local search, which always terminates because the number of patterns is finite and we continue only if we found a better pattern than the previous one. 
\begin{remark}
In~\Cref{algo::explore_good_neighbors}, the exploration runs along segments, but it could also be made in the reversed order (loop on the contracts and selection of the worst active / best nonactive segments). It appears in the numerical tests that the latter option is less efficient.
\end{remark}

The local search ends up with a local optimum in the sense that there is no neighbor (achievable by  \texttt{exploreGoodNeighbors}) that produces a better solution. Then, to improve this solution, we need to consider
a
larger neighborhood. This is the object of the procedure \texttt{MIQP\_restart}, described in~\Cref{algo::MIQP_restart}, in which we construct a small MIQP, fixing binary variables, except for the following ones:
\begin{enumerate}[label=(\roman*)]
  \setlength\itemsep{-.1\baselineskip}
\item $\gamma^S$ segments: for such a segment $s$, the variables $z_{s,w}$ in~\eqref{eq::beta_MIQP} become free for all $w \in 0\hdots W$; in other words,
the whole row $s$ in the pattern may be changed;
\item $\gamma^W$ contracts: for such a contract $w$, the variables $z_{s,w}$ in~\eqref{eq::beta_MIQP} become free for all $s \in [S]$;  in other words,
   the whole column $w$ in the pattern may be changed;
 \item every variable $z_{s'w'}$ with $s'\neq s$ and $w'\neq w$
   is made free with probability $\sigma \in [0,1]$. 
\end{enumerate}
This restart procedure uses a pattern as input and ends either with this pattern or a better one if the MIQP has found such a pattern.

\begin{algorithm}[!ht]
\caption{\texttt{MIQP\_restart}}\small
\begin{algorithmic}[1]
\Require $A$ \Comment{initial pattern}
\State Select $\gamma^S$ segments, $\gamma^W$ contracts and coefficients $(s,w)$ with probability $\sigma$
\State Constrain $z$ to be equal to $A$, except for the chosen segments, contracts and coefficients
\State $A^*,x^*,\varphi^* \leftarrow$ optimum of \eqref{eq::beta_MIQP} with the additional constraints on $z$.
\State\Return $A^*,x^*,\varphi^*$
\end{algorithmic}
\label{algo::MIQP_restart}
\end{algorithm}

The complete heuristic (\Cref{algo::QSPC}), which we call~\emph{Quadratic Search on Price Complex} (\texttt{QSPC}), alternates between the local search and the restart phase until no progress is made, i.e., several iterations do not have produced any improvement.

\begin{algorithm}[!ht]
\caption{Quadratic Search on Price Complex (\texttt{QSPC})}\small
\begin{algorithmic}[1]
\Require $A,x_A,\varphi_A =\pi^{quad}(x_A;\beta),r_{max}$ \Comment{$x_A$ optimum on the initial pattern $A$}
\State $r\leftarrow 0$
\State $A^*,x^*,\varphi^* \leftarrow A,x_A,\varphi_A$\Comment{$A^*$ will be the best pattern found} 
\While{$r<r_{max}$} \Comment{$r=\#$ restarts without improvement}
\If{$r=0$}
\State $opt_{loc}\leftarrow$ false
\While {$opt_{loc}$ is false}\Comment{Until a local optimum is found}
\State $A',x_{A'},\varphi_{A'}\leftarrow \texttt{exploreGoodNeighbors}(A^*,x^*,\varphi^*)$
\State $opt_{loc} \leftarrow (A'= A^*)$
\State $A^*,x^*,\varphi^*\leftarrow A',x_{A'},\varphi_{A'}$
\EndWhile
\EndIf
\State $A',x_{A'},\varphi_{A'}\leftarrow \texttt{MIQP\_restart}(A^*)$
\If{$A'=A^*$} \State $r \leftarrow r+1$
\Else \State $A^*,x^*,\varphi^*,r \leftarrow A',x_{A'},\varphi_{A'},0$ \Comment{Update best pattern}
\EndIf
\EndWhile
\State\Return $A^*,x^*,\varphi^*$
\end{algorithmic}
\label{algo::QSPC}
\end{algorithm}

\section{Application to Electricity Pricing}\label{sec-numerical}
\subsection{Instance definition}
In the numerical tests, we consider an electricity pricing problem: a power retailer has $W=4$ different contracts that need to be optimized, each one depending on $H=3$ coefficients (peak/off-peak/fixed part)\footnote{Here, we call ``peak period'' the interval 8am -- 8pm. The other twelve hours defines the ``off-peak'' period.}. These contracts mimic the most common type of contracts existing in the French power markets, and are listed in~\cref{tab::contracts}. To evaluate the costs $C_{sw}$, we use the methodology from the French regulator which consists in summing the different costs such as electricity production cost, taxes, transport and distribution network charges or commercial margin, see e.g.~\cite[Figure 1]{CRE_summing_costs}. The costs of electricity production are evaluated as the average of historical market prices to represent that the retailer buys the energy for its customers on power exchanges over the whole year. Costs are therefore not reflecting the hourly variabilities of electricity market prices but this is coherent with our approach which is not a dynamic time pricing but a fixed one.

\begin{table}[!ht]
\centering\small
\begin{tabular}{|r|c|c|l|}
\hline
1 & Base & \multirow{2}{*}{Standard}  & \multirow{2}{*}{Low cost offers (digital-only customer services)}\\
2 & Peak/Off peak&  & \\
\hline
3 & Base & \multirow{2}{*}{Green}\footnote{This type of contract provides power generated from renewable source such as on-shore wind and the retailer has to provide guaranties of origin which induces additional costs.} & Higher costs, but preferred by some segments\\
4 & Peak/Off peak &  &  (higher reservation bill)\\
\hline
\end{tabular}
\caption{Contracts used in the instances}
\label{tab::contracts}
\medskip\small
Each offer has a base load version (no price difference between peak and off-peak periods) and a version with different prices at peak and off-peak periods, making a total of 10 contracts.
\end{table}

\begin{figure}[!ht]
\centering
\begin{subfigure}{0.55\linewidth}
\includegraphics[width=1\linewidth,clip=true, trim = 4cm 16.5cm 3cm 4cm]{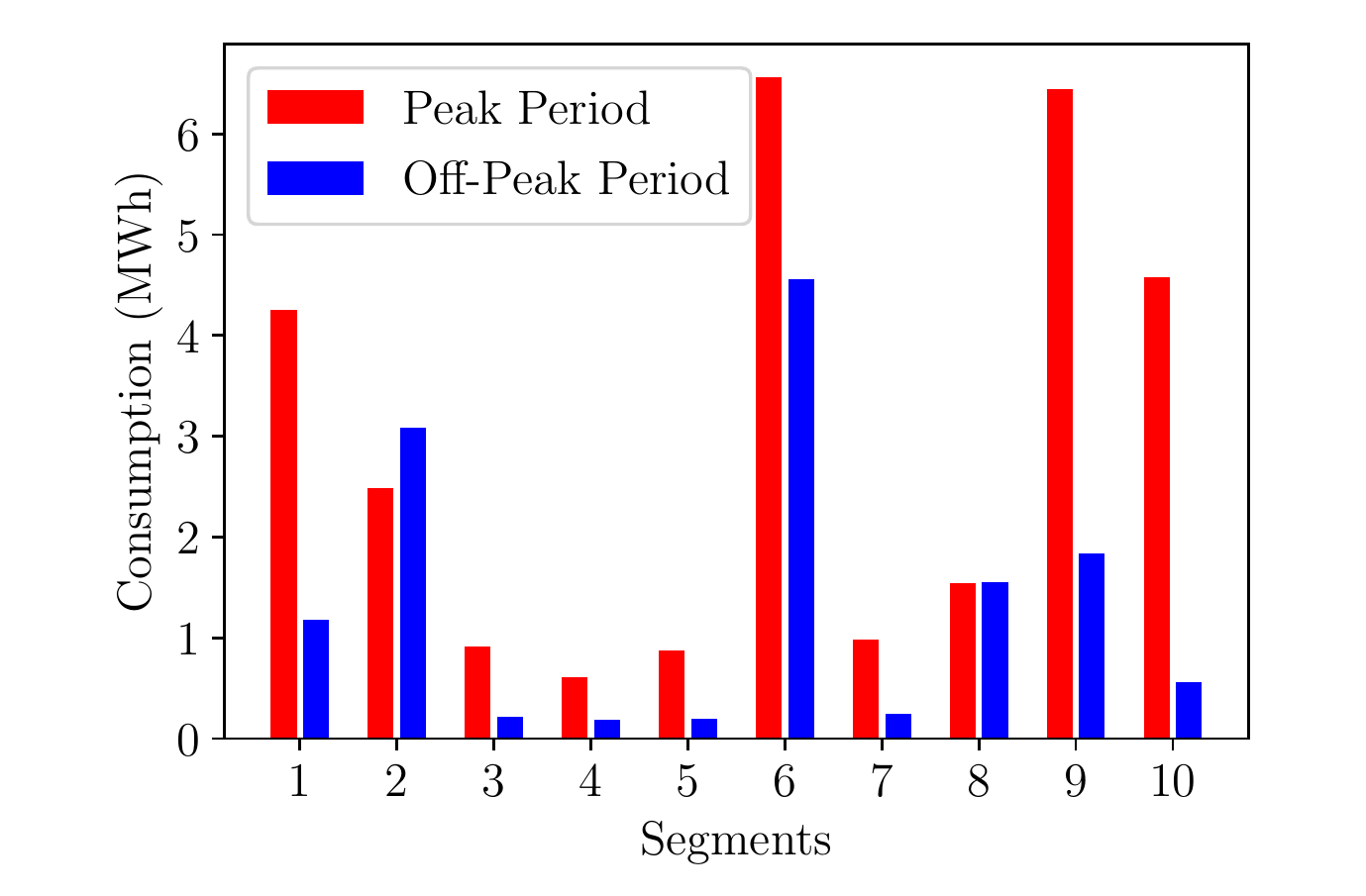}
\caption{Nominal consumption of segments, over one year. 
For each segment, the consumption is separated into the Peak period and the Off-peak period.}
\end{subfigure}
\begin{subfigure}{0.35\linewidth}
    \includegraphics[width=0.9\linewidth,clip=true, trim = 3cm 8cm 1.5cm 4cm]{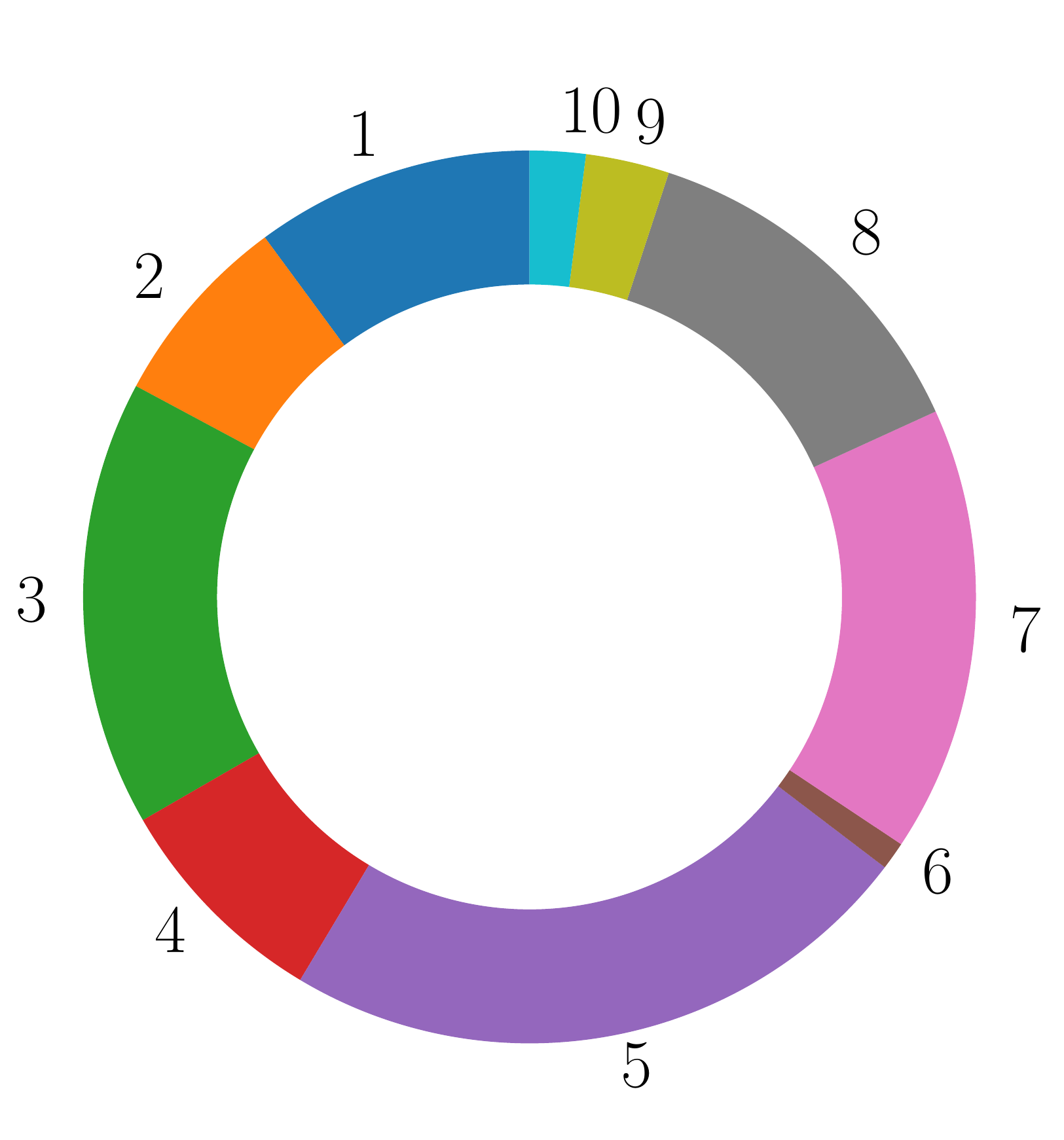}
    \caption{
Weights of segments. For each segment, the size of the section corresponds to the proportion of users in this segment.}
\end{subfigure}
\caption{Clustering for 10 segments.}
\label{fig::consumption}
\end{figure}

Concerning the customers, a thousand load curves (obtained by the \texttt{SMACH} simulator of EDF, see \shortciteA{Huraux_2015}) represent various power consumption profiles and mimic the entire French population, taking into account different household compositions, locations, and electrical equipments.  To construct our set of instances, we used the $k$-means algorithm to obtain $S$ clusters (segments), where $S=10$. In this way, customers that have similar consumption profile and contract preferences are aggregated in the same cluster. \Cref{fig::consumption} displays the nominal consumption after the clustering process, i.e., the aggregated year-based consumption that a typical customer of the segment is expected to consume when he faces a constant price. Segments 6 and 9 correspond to consumers with high electricity consumption and typically have individual houses with full-electric equipments and especially electrical heating. By contrast, segments 3, 4 and 5 are low energy consumers which are small households without electrical heating and cooking. Segments 1, 9 and 10 have a highly differentiated peak/off-peak profile compared with the others which consume in a more regular way. For peak/off-peak contracts, we suppose that each customer can shift a part of the consumption from peak period to off-peak period (\emph{load shifting}). Here, we suppose that 15\% of the nominal peak consumption can be shifted to off-peak periods.
Moreover, we suppose that the green preference is cast into three categories: highly / mediumly / lowly eco-friendly. This corresponds to an additional utility of 4\% / 2\% / 0\% of their bill computed with regulated prices\footnote{The instances are not intended to fully depict the reality of the market, but they are already enough rich to deliver some useful insights on the effectiveness of the model.}. For instance, segments 3, 4 and 5 have similar nominal consumptions (see~\Cref{fig::consumption}) but different green preferences (see~\Cref{fig::rep}).
In this study, we consider 6 competitors' offers, defined with real prices that can be found in the French market. These offers are depicted in~\Cref{tab::competitors}.
\begin{table}[!ht]
\centering\small
\begin{tabular}{|r|r|>{\columncolor{green!40}}r|r|>{\columncolor{green!40}}r|r|r|}
\hline
Competitors & 1 & 2 & 3 & 4 & 5 & 6\\  
\hline
\hline
Peak (\euro/kWh)& \multirow{2}{*}{0.174} &   & 0.1840 & 0.19 &  \multirow{2}{*}{0.166} & 0.23\\
\hhline{|-|~|>{\arrayrulecolor{green!40}}->{\arrayrulecolor{black}}|--|~|-|}
Off peak (\euro/kWh)& & \multirow{-2}{*}{ 0.1819} & 0.147 & 0.155 & & 0.135\\
\hline
Fixed portion (\euro)& 136 & 136 & 144 & 144 & 148 & 141\\
\hline
\end{tabular}
\caption{Competitors prices. Contract 2 and 4 are green contracts}\label{tab::competitors}
\end{table}

\subsection{Numerical analysis}

\begin{table}[!ht]
\centering
\begin{subtable}{\linewidth}
 \centering\small
\begin{tabular}{|r|r|>{\columncolor{green!40}}r|r|>{\columncolor{green!40}}r|}
\hline
Contract & 1 & 2 & 3 & 4\\  
\hline
\hline
Peak (\euro/kWh)&  \multirow{2}{*}{0.166} &  & 0.1768 & 0.2215\\
\hhline{|-|~|>{\arrayrulecolor{green!40}}->{\arrayrulecolor{black}}|--|}
Off peak (\euro/kWh)& & \multirow{-2}{*}{0.1819} & 0.1607 & 0.1391\\
\hline
Fixed portion (\euro)& 148 & 136 & 136.29 & 120\\
\hline
\end{tabular}
\caption{Optimal prices with deterministic setting}
\label{tab::optimal_prices_det}
\end{subtable}\\
\begin{subtable}{\linewidth}
 \centering\small
\begin{tabular}{|r|r|>{\columncolor{green!40}}r|r|>{\columncolor{green!40}}r|}
\hline
Contract & 1 & 2 & 3 & 4\\  
\hline
\hline
Peak (\euro/kWh)&  \multirow{2}{*}{0.1693} &  & 0.1863 & 0.1895\\
\hhline{|-|~|>{\arrayrulecolor{green!40}}->{\arrayrulecolor{black}}|--|}
Off peak (\euro/kWh)& & \multirow{-2}{*}{0.1834} & 0.1491 & 0.1626\\
\hline
Fixed portion (\euro)& 133.7 & 129.29 & 122.95 & 128.19\\
\hline
\end{tabular}
\caption{Optimal prices with quadratic regularization of intensity $\beta = 0.2$}
\label{tab::optimal_prices_quad}
\end{subtable}
\caption{Optimal prices}
\label{tab::optimal prices}
\end{table}

\Cref{tab::optimal_prices_det} shows the optimal prices for the deterministic case, and~\Cref{fig::rep_det} shows the corresponding distribution over the contracts. As previously explained on a theoretical example (\Cref{fig::comparison_regularization}), the deterministic model adjusts the prices so that many customers face two contracts with equal utilities. This can be viewed in~\Cref{fig::rep_det} where the hatched bars represent the ties in the choice. In~\Cref{tab::optimal_prices_det}, we have the extreme case where the second retailer's contract contends exactly the same coefficients as the second competitors' offer. We also noticed that the segments who naturally favor green energy (segments 5, 8 and 9) chose a green contract and that some segments are attributed to the competitors (such as segments 2, 6, 8 and 9) as they must be too costly for the retailer.

\begin{figure}[!ht]
\centering
\begin{subfigure}{1\linewidth}
\centering
\includegraphics[width=0.85\linewidth,clip=true, trim = 0.6cm 14.8cm 1cm 7cm]{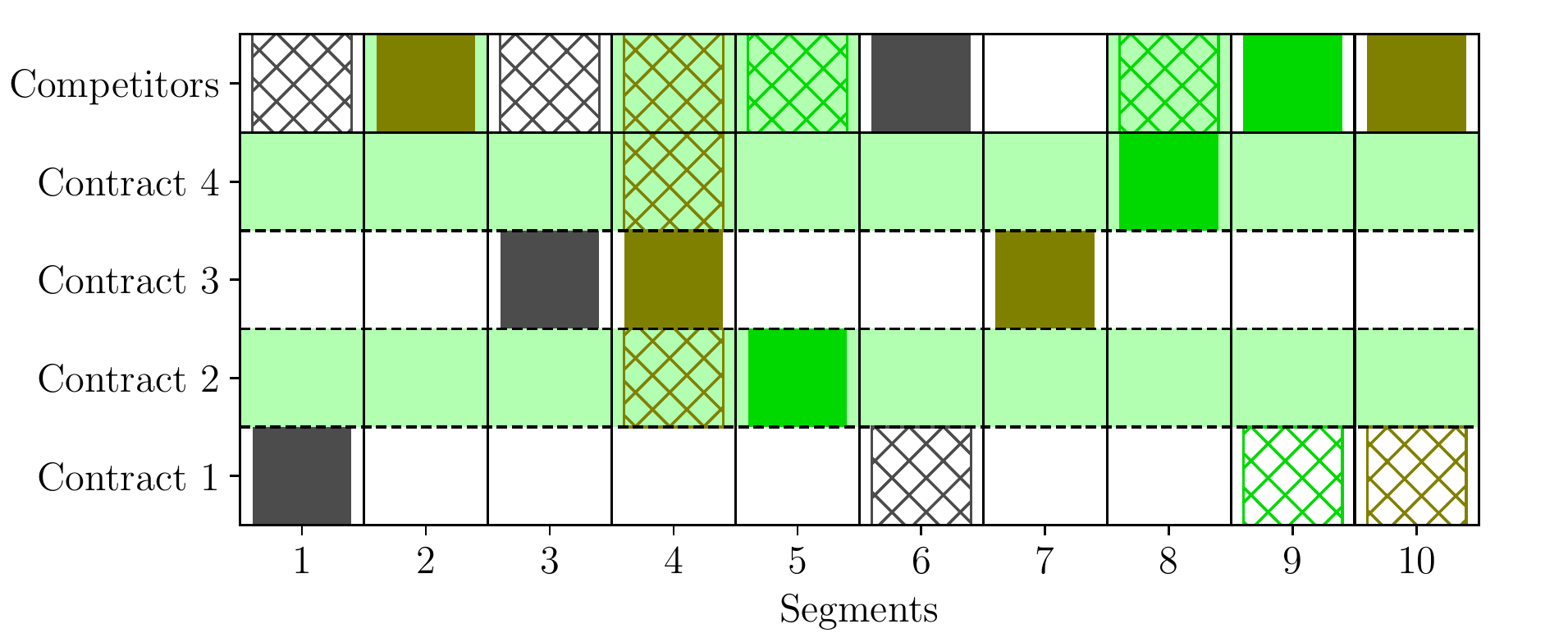}
\caption{Optimal customers' distribution with deterministic setting. \\
A hatched bar means that the segment had the same utility as the chosen contract, but favors the retailer by choosing the one with the highest profit value (it could be a competitors' offer).}
\label{fig::rep_det}
\end{subfigure}
\begin{subfigure}{1\linewidth}
\centering
\includegraphics[width=0.85\linewidth,clip=true, trim = 0.6cm 14.8cm 1cm 7cm]{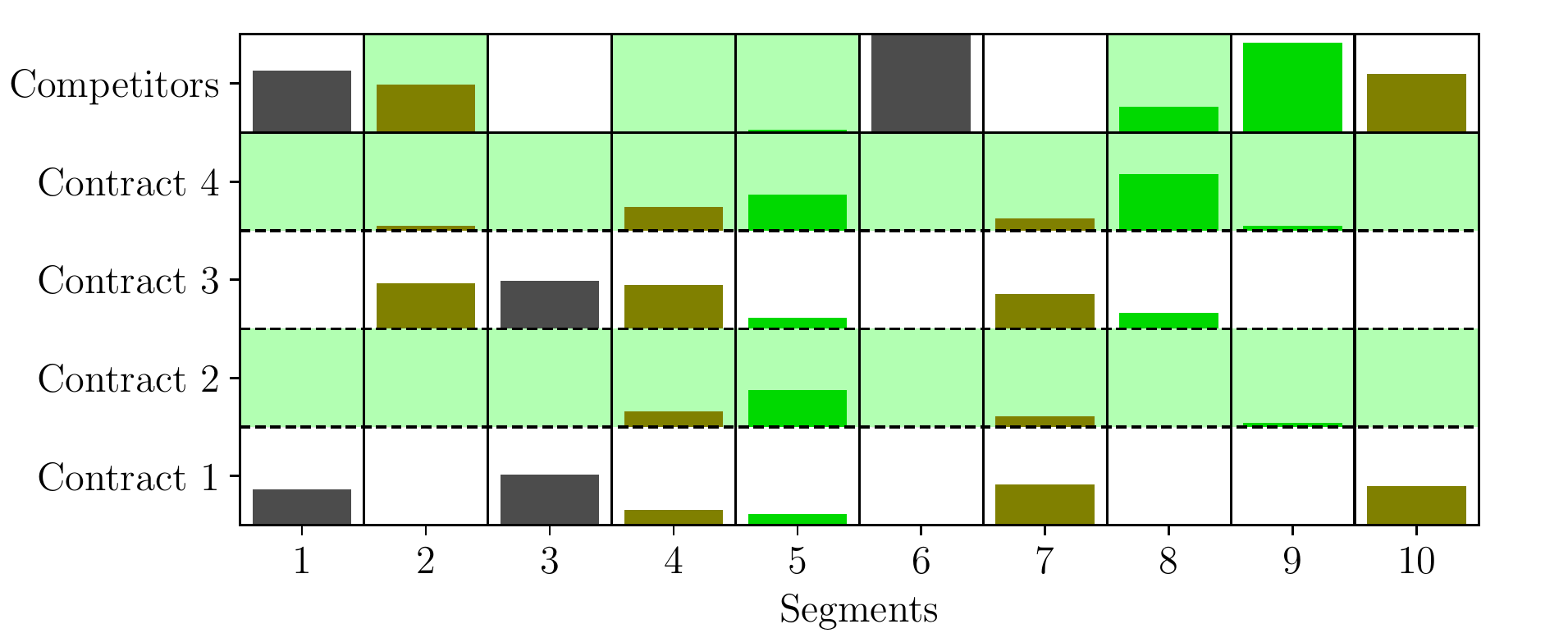}
\caption{Optimal customers' distribution with quadratic regularization of intensity $\beta=0.2$.\\
The size of the bar defines the probability of choices, i.e., a bar taking a fourth of the rectangle height represents a choice probability of 25\%.}
\label{fig::rep_quad}
\end{subfigure}
\begin{subfigure}{1\linewidth}
\centering
\includegraphics[width=0.85\linewidth,clip=true, trim = 0.6cm 14.8cm 1cm 7cm]{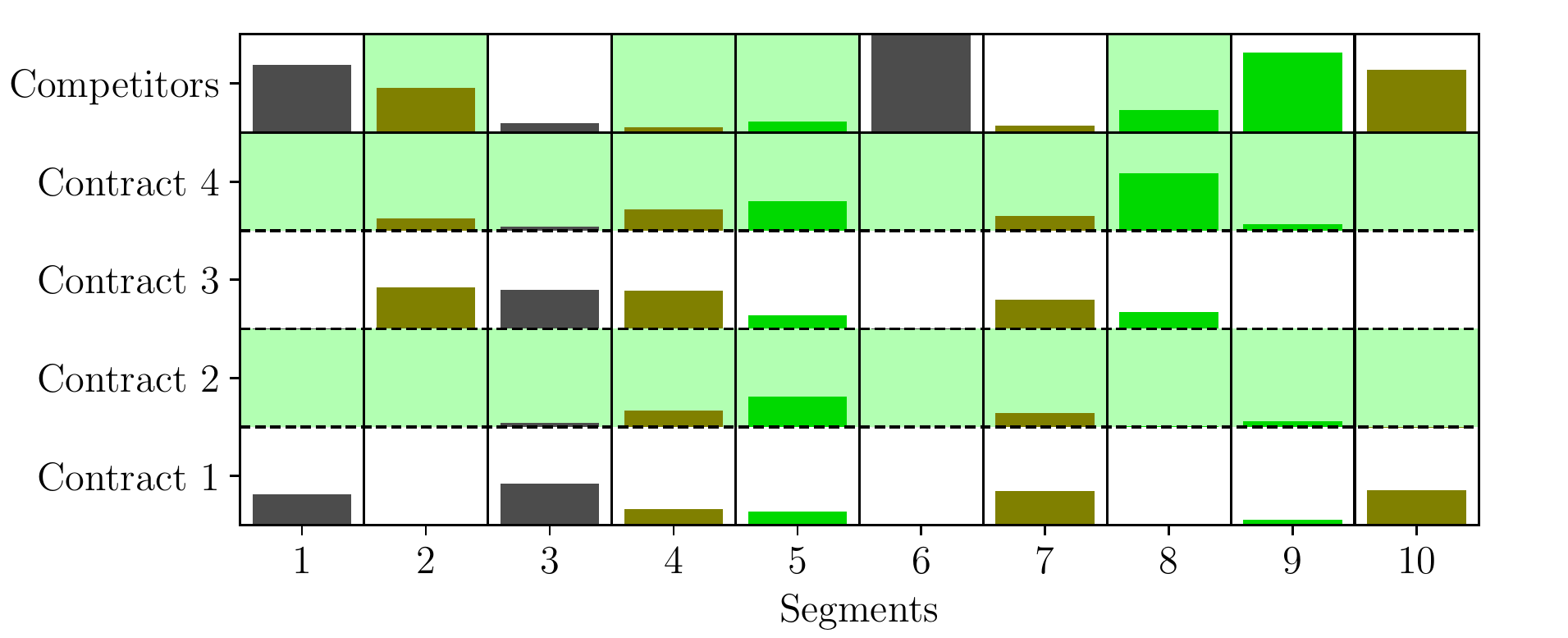}
\caption{Customers' logit distribution with intensity $\beta=0.8/e$ for prices of~\Cref{tab::optimal_prices_quad}.\\}
\label{fig::rep_log}
\end{subfigure}
\caption{Optimal customers' distributions. \\{\small Green contracts are displayed with a green-filled rectangle. Decisions of highly (resp. mediumly / lowly) eco-friendly clusters are displayed with green (resp. brown / gray) bars. The six offers of~\Cref{tab::competitors} are summed up into the first line, where only the best competitors' offer is displayed.}}
\label{fig::rep}
\end{figure}

We also displays the results for the regularized case $\beta=0.2$. In this example, this choice of $\beta$ appears to be close from the worst case from a retailer's point of view (see~\Cref{fig::range_beta}) and is, in a sense, robust to any choice of $\beta$ value. We observe that the optimal price grid (\Cref{tab::optimal_prices_quad}) is somehow different from the optimistic one. Every contract has a lower fixed part in the regularized case compared with the deterministic case, but the variable portions can be either lower or greater. Concerning the customers' distribution along the contracts, we observe that the choices are globally preserved in the sense that every deterministic decision stays privileged in the regularized case. Let's notice that high consumption segments (segments 6 and 9) and highly differentiated peak/off-peak (segments 1, 9 and 10) are for a great part not favored by the retailer and let to competitors for a high proportion. On the other side, our retailer manages to attract green segments. In order to compare with logit approach, we also show the logit customers' distribution computed using the optimal quadratic prices. This distribution is very similar to the quadratic case, see~\Cref{fig::rep_quad} and~\Cref{fig::rep_log}. The main difference lies in the small probabilities: the logit choice is slightly more spread on the different contract, but the probabilities stay highly comparable. We refer to~\Cref{app::estimates} where we provide a more detailed comparison between quadratic and logit approaches and develop metric estimates to quantify the deviation between the two models. 

\begin{figure}[!ht]
\centering
\includegraphics[width=0.68\linewidth]{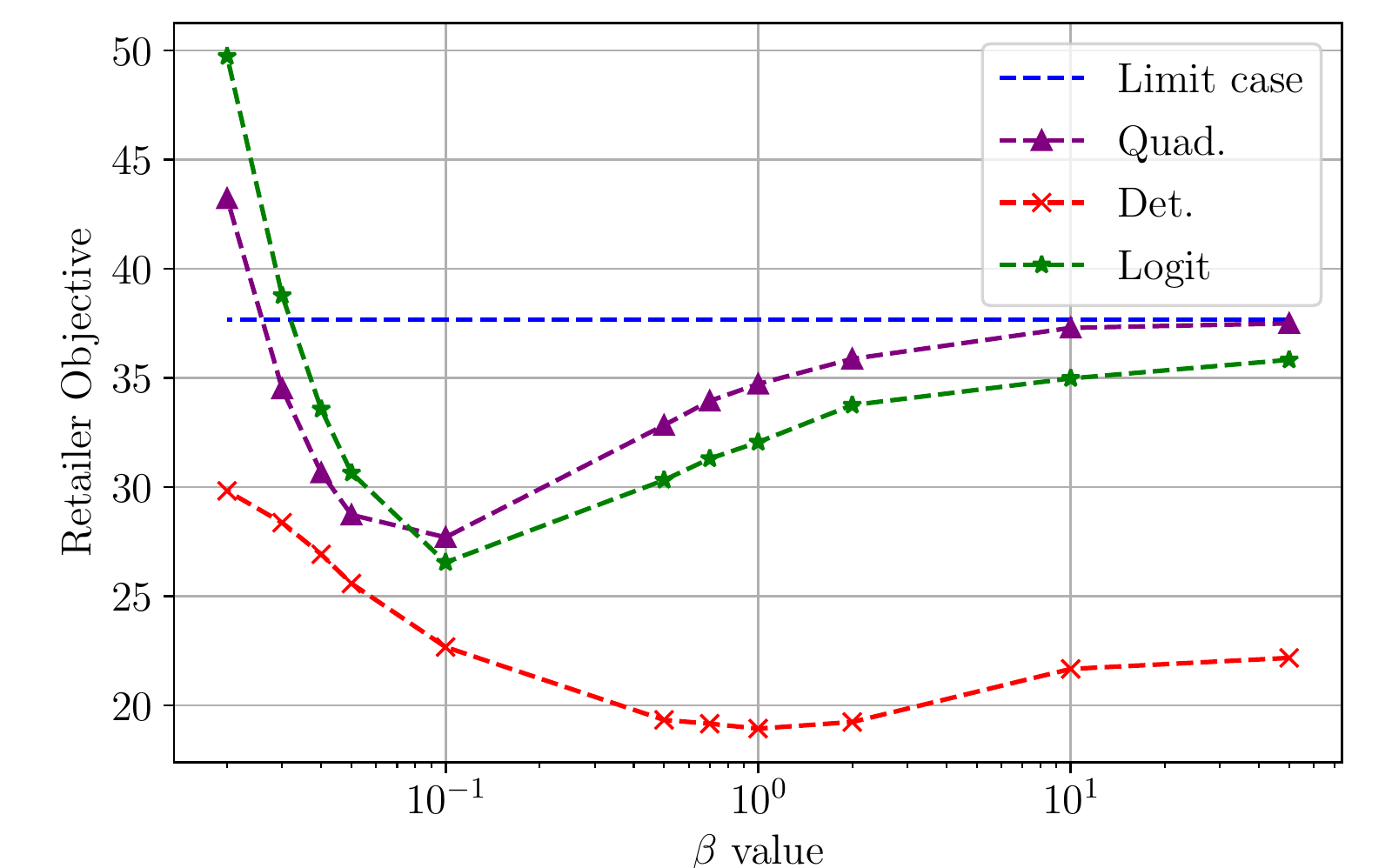}
\caption{Optimal value as a function of the rationality parameter $\beta$.\\{\small
We display the results for the model under logit response (Logit) and for the model under quadratic response (Quad.). In addition, we display the objective value obtained by applying the optimal prices of the deterministic model (\Cref{tab::optimal_prices_det}), assuming a quadratic response of the customers (Det.). }
}
\label{fig::range_beta}
\end{figure}
To analyze the impact of the regularization, \Cref{fig::range_beta} draws the profit function (retailer objective) as a function of the regularization intensity. About the logit and quadratic model, the result for small $\beta$ values is quite intuitive: with customers randomly reacting, the company can impose very high prices since there will always be some consumers taking its contracts. Hence, the company's profit becomes infinite as $\beta \rightarrow 0$. For the company, having deterministic customers is more beneficial since the price can be adjusted to perfectly fit the population behavior. This can be interpreted as the result of moral hazard: the randomness in the followers decision negatively impacts the leader revenue.

We also display the objective function value obtained by fixing the price to the optimal prices found in the deterministic setting, i.e., supposing $\beta = \infty$, but recalculating the response and the optimal objective value in the uncertain context (finite value $\beta$). We see that the objective value is far below the optimal quadratic solution (indeed, around $40$\% of revenue are lost for large values of $\beta$). This highlights that the deterministic solution is unstable, and not robust to uncertainty, see comments in~\Cref{sec::comp_logit}. It is then necessary to consider regularized consumers behavior to obtain a reliable menu of offers.

\begin{remark}
For completeness, one can find in~\Cref{app::comparison_solvers} a numerical comparison of the \texttt{QSPC} solver with other methods. This study is performed on various instance sizes and compares the proposed method with direct resolution, MPEC solvers (via nonlinear reformulations), and on-the-shelf heuristics. In particular, we could solve instances of substantial size (10 contracts, 50 segments) in a reasonable time with a MIP gap tolerance of $3$\%.
\end{remark}

\section{Conclusion}
We explored an extension of the unit-demand envy-free pricing problem, in which the customer invoice is determined by multiple price coefficients. We first analyzed a bilevel programming model, assuming a fully deterministic behavior of customers (every customer takes only one contract, maximizing her utility). This is inspired by known models in the case of a unique price coefficient. Such bilevel problems reduce to mixed linear programming, allowing one to solve instances of intermediate size to optimality. However, the assumption of deterministic behavior is not realistic, at least for the class of electricity pricing problems that motivate this work.
So, we developed a new, alternative model, based on a quadratic
regularization, which combines tractability and realism. We demonstrated that the lower response map of this quadratic model is characterized by a polyhedral complex, and using this geometrical property, we designed a heuristic which showed its efficiency in terms of optimality and time on our data set. We finally analyze the behaviors of the three models (deterministic, logit and quadratic) on a use case and highlight once again the need of a (tractable) probabilistic choice model to avoid unrealistic solutions.

Several extensions may be considered to further improve the realism. In particular, throughout the paper, competitors are supposed not to adjust their prices to the strategy of the company (static competition).
Relaxing this assumption would imply to consider a Nash equilibrium between leaders (multi-leader-common-follower games). In particular, this has been studied by~\shortciteA{Leyffer_2010}, where an application to the electricity market is also the main motivation. Nonetheless, even considering the deterministic case (perfect knowledge and purely rational decision), only stationary points (not necessarily local solutions) can be numerically found in general, and for relatively small instance size. Besides, we also suppose that and customers immediately react to the prices (no switching cost). Modeling such features would lead to dynamic games, 
increasing a lot the computational time, and making the above numerical study intractable.


\section*{Acknowledgments}

We thank Riadh Zorgati and Pedro Suanno for their fruitful discussions on this subject, and the EDF team of developers of the \texttt{SMACH} simulator, who provided load curves based on consumers data.
We also thank the reviewers for their detailed comments and for references, which helped us to improve this work.

\bibliography{./biblio.bib}
\appendix

\clearpage
\section{Proof of~\Cref{prop::convergence_cell}}\label{app::proof_convergence_cell}
\begin{lemma}\label{lemma::convergence_beta}Consider two sequences of polyhedra $P^+_\beta$ and $P^-_\beta$ defined as $P^\pm_\beta := \left\{x\in X: A x \leq b \pm \beta^{\shortminus 1} e\right\}$ ($e$ is the all-ones vector), and the limit case $P:=\left\{x\in X: Ax \leq b\right\}$. Then, $P^+_\beta \xrightarrow[\beta]{} P$ and $\lim_{\beta}P^-_\beta \subseteq P$. Moreover, if $\Int(P)\neq \varnothing$, $P^-_{\beta}\xrightarrow[\beta]{}P$.
\end{lemma}
\begin{proof}
Throughout the proof, we consider a sequence $(\beta_n)$ converging to $\infty$, and the notation $P^\pm_n$ has to be understood as $P^\pm_{\beta_n}$.
  
The two monotone sequences have a limit: $\lim_n P^+_n = \bigcap_n P^+_n$ and $\lim_n P^-_n = \bigcup_n P^-_n$, see~\cite[Exercise 4.3]{RW_2009}, it remains to prove that this limit coincides with $P$. Two first inclusions come with the definition of the sequences:
$\lim_{n}P^-_n \subseteq P$ and $P \subseteq \lim_{n}P^+_n$.

Let us consider $x \notin P$. If $x\in X\backslash P$, then there exists a row $i$ such that $A_i x = b_i + \epsilon$ where $\epsilon > 0$. Therefore, for $\beta_n\geq \epsilon^{\shortminus 1}$, $x \notin P^+_n$. Otherwise, if $x\notin X$, $x$ cannot be in any $P^+_n$. In any case, $x\notin P \Rightarrow x\notin \lim_n P^+_n$, and therefore $\lim_n P^+_n \subseteq P$.

We now assume that $\Int(P)\neq \varnothing$. For any given $x\in P$, let us define the sequence $x_n:=\Proj_{P^-_n}(x)$. Since the $P^- \nearrow$, the distance $\|x_n-x\|$ is a decreasing sequence bounded from below by 0 and converges to a distance $d \geq 0$. Suppose now that $d>0$, then for any unitary vector $u$, $x+du \notin P^-_n, n\in \bbN$. Besides, there exists $0\leq d'\leq d$ and a unitary vector $v$ such that $x+d'v \in \Int(P)$.  Defining $y = x+d'v$, we obtain that $y\in\Int(P)$ and $y\notin P^-_n, n\in \bbN$. As it belongs to the interior of $P$, $Ay \leq b -\epsilon e,\epsilon > 0$ and for any $\beta_n \geq \epsilon^{\shortminus 1}$, $y\in P^-_n$. This yields a contradiction: $d$ must be equal to 0, and therefore $x_n\to x$. 
To conclude, for any $x\in P$, we can exhibit a sequence of points $x_n\in P^-_n$ converging to $x$, so $P\subseteq\lim_n P^-_n$.\end{proof}

Using~\Cref{lemma::convergence_beta}, one can obtain the following inclusions:
$$
\begin{aligned}
\limsup_\beta \overline{X}(A;\beta) &= \limsup_{\beta}\left(\overline{X}^0(A;\beta)\cap\overline{X}^1(A;\beta)\right) \\&\subseteq \lim_{\beta}\overline{X}^0(A;\beta)\cap\lim_\beta\overline{X}^1(A;\beta)\subseteq \overline{X}^0(A;\infty)\cap\overline{X}^1(A;\infty)\enspace.
\end{aligned}
$$
Moreover, if $\Int\left(\overline{X}(A;\infty)\right) \neq \varnothing$, then $\lim_\beta \overline{X}^0(A;\beta) = \overline{X}^0(A;\infty)$, see~\Cref{lemma::convergence_beta}. Besides, $\overline{X}^0(A;\infty)$ and $\overline{X}^1(A;\infty)$ cannot be separated, and therefore $\overline{X}^0(A;\beta)\cap\overline{X}^1(A;\beta) \xrightarrow[\beta]{} \overline{X}^0(A;\infty)\cap\overline{X}^1(A;\infty)$, see~\cite[Theorem 4.32c]{RW_2009}.

\section{Complexity}\label{sec-complexity}
\shortciteA{Guruswami_2005} proved that the deterministic model is APX-hard (see~\shortciteA{Paschos_2009} for a description of this class). Using this result, we prove that the quadratic case is also APX-hard:
\begin{prop}
The problem \eqref{eq::bilevel_quad} is APX-hard, even in the single-attribute setting and without price constraints.
\end{prop}
\begin{proof}
Reusing the same polynomial transformation (and the same notations) as in~\shortciteA{Guruswami_2005}, we claim the existence of a sufficiently large parameter $\beta$ ($\beta \geq 8(n+m)$) such that the quadratic optimal value is not far from the deterministic one i.e., $\vert \val(q\beta \trt BP) - \val(o \trt BP)\vert  \leq 1/4$.\\
First, it can be noticed that the optimal prices cannot be any values: for any product,
\begin{itemize}
    \setlength\itemsep{.1em}
    \item if the price is  in $]\frac{2}{\beta}, 1-\frac{2}{\beta}[$, then customers having a null reservation bill for the contract will have no chance to purchase it and customers having reservation bill of 1 or 2 will purchase it with probability 1. So the company has more interest in setting the price at $1-\frac{2}{\beta}$.
    \item With the same logic, if the price is  in $]1+\frac{2}{\beta}, 2-\frac{2}{\beta}[$, then the company has more interest in setting the price at $2-\frac{2}{\beta}$.
    \item If the price is less than $\frac{2}{\beta}$, the profit made by the company with this contract is less than 1/4, so setting the price to $1-\frac{2}{\beta}$ is more beneficial.
    \item Finally, a price greater than $2+\frac{2}{\beta}$ does not make any profit.
\end{itemize}
For an optimal solution, the price values can only be in $[1-\frac{2}{\beta}, 1+\frac{2}{\beta}] \cup [2-\frac{2}{\beta}, 2+\frac{2}{\beta}]$.
Taking the optimal quadratic prices and rounding them to obtain a price vector of values 1 or 2 provides a price vector for the deterministic problem with a value closed to the quadratic optimum i.e., $\val(o \trt BP) \leq \val(\beta \trt BP) - \frac{2}{\beta}(n+m)$.

For the converse, taking the optimal deterministic solution (we know that the prices can only be 1 or 2) and subtracting $\frac{2}{\beta}$ to each price gives a quadratic solution with objective value closed to the deterministic optimum i.e., $\val(\beta \trt BP) \leq \val(o \trt BP) - \frac{2}{\beta}(n+m)$.
\begin{figure}[!ht]
\centering
\includegraphics[width=0.85\linewidth, clip = true, trim = 1.5cm 0.3cm 0cm 0.4cm]{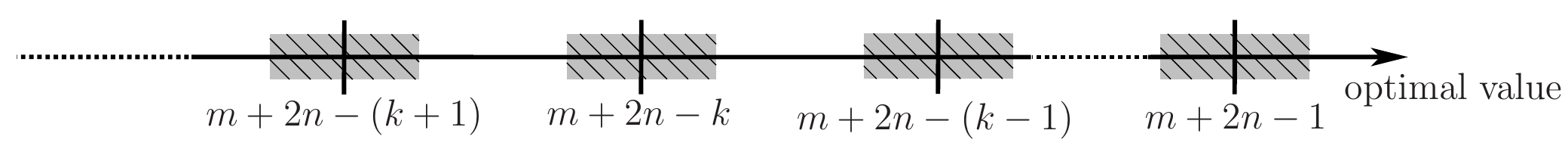}
\caption{Representation of the objective value for the transformation of Guruswami et al.}
\medskip
\small
Deterministic optimum is integer and the quadratic one lies in a small interval centered on it (hashed zones).
\end{figure}

Computing the quadratic optimum for $\beta \geq 8(n+m)$ and rounding it gives us the deterministic optimum. Thus, the quadratic case is at least as hard as the deterministic case, which was proved to be APX-hard.
\end{proof}
\begin{remark}
The structure of this specific instance allows us to exhibit a threshold from which the quadratic model is a sufficiently good approximation for the deterministic model. In a more general case, even if we have established the convergence of the quadratic model to the deterministic one, we are not able to provide such a threshold.
\end{remark}

\section{Metric estimates to compare logit and quadratic regularization}\label{app::estimates}
\begin{prop}
Consider a segment $s$ facing $W+1$ disutilities $V_{s0}, \hdots, V_{sW}$ sorted in ascending order. For a given $\beta>0$, we denote by $(y^{quad}_{sw})_w$ the quadratic response (computed with a parameter $\beta'=\beta e/4$) and by $(y^{log}_{sw})_w$ its logit analog (computed with $\beta$). Then, 
\begin{equation}
\text{If } y^{quad}_{sw} = 0,\text{ then } y^{log}_{sw} \leq \gamma_w:= \left(1+ we^{\frac{8}{we}}\right)^{-1}\quad (\leq 1/9)
\label{eq::bounds_gamma}
\end{equation}
Conversely,
\begin{equation}\text{If } y^{log}_{sw} \leq \eta^W_w := \left(W+1+w(e^{\frac{8}{e}}-1)\right)^{-1},\text{ then } y^{quad}_{sw} = 0
\label{eq::bounds_eta}
\end{equation}
\label{prop::metric_compare}
\end{prop}
\begin{proof}
Suppose that $y^{quad}_{sw} = 0$, then from~\Cref{prop::construction_sol_quad}, $V_{sw} \geq c_{sw} = \frac{1}{w}\left[\frac{8}{e\beta}+\sum_{k=0}^{w-1}V_{sw}\right]$ and thus
$$\exp\left(\frac{8}{we} -\beta V_{sw}\right) \leq \exp\left(-\frac{1}{w}\sum_{k=0}^{w-1}\beta V_{sk}\right)\leq \frac{1}{w}\sum_{k=0}^{w-1}e^{-\beta V_{sk}}\enspace,$$ where the latter inequality is obtained by convexity of the exponential. We then deduce that $\gamma_w^{\shortminus 1} e^{-\beta V_{sw}} \leq \sum_{k=0}^{w}e^{-\beta V_{sk}}$.
Using the logit expression gives us the desired result.

Suppose that $y^{log}_{sw} \leq \eta$ for a given $\eta$. We exploit the ascending sort on $V$ in the logit expression to obtain
$$
\eta \geq \frac{e^{-\beta V_{sw}}}{\sum_{k=0}^{w-1}e^{-\beta V_{sk}} + \sum_{k=w}^{W}e^{-\beta V_{sk}}} \geq \frac{e^{-\beta V_{sw}}}{\sum_{k=0}^{w-1}e^{-\beta V_{s0}} + \sum_{k=w}^{W}e^{-\beta V_{sw}}} \enspace .$$
Continuing the simplifications, $\eta^{\shortminus 1}\leq w e^{-\beta (V_{s0}-V_{sw})} + (W-w+1)$ and therefore $$V_{sw} \geq V_{s0} + \frac{1}{\beta}\log\left(\frac{\eta^{-1} - (W-w+1)}{w}\right)\enspace.$$
Finally, taking $\eta = \eta^W_w$ implies that $V_{sw} \geq V_{s0} + \frac{8}{e\beta}$, insuring that $y^{quad}_{sw} = 0$.
\end{proof}
The technical \Cref{prop::metric_compare} shows that there is a common convergence speed to the deterministic behavior: in fact, for any value of $\beta$, if we have no ``quadratic chance'' to choose a contract $w$ then we have a very little logit probability to choose $w$. The converse applies but it depends on the total number of contracts; in the logit version, the probability depends on the whole set of contracts whereas the quadratic version does not care of the contracts that have a very large disutility. It is important to note that the bounds $\gamma_w$ and $\eta^W_w$ in \eqref{eq::bounds_gamma} and \eqref{eq::bounds_eta} are valid for any value of $\beta$.
\begin{figure}[!ht]
  \centering
    \vspace{-.2cm}
    \includegraphics[width=0.45\linewidth, clip = true, trim = .2cm 1.2cm .5cm 0.6cm]{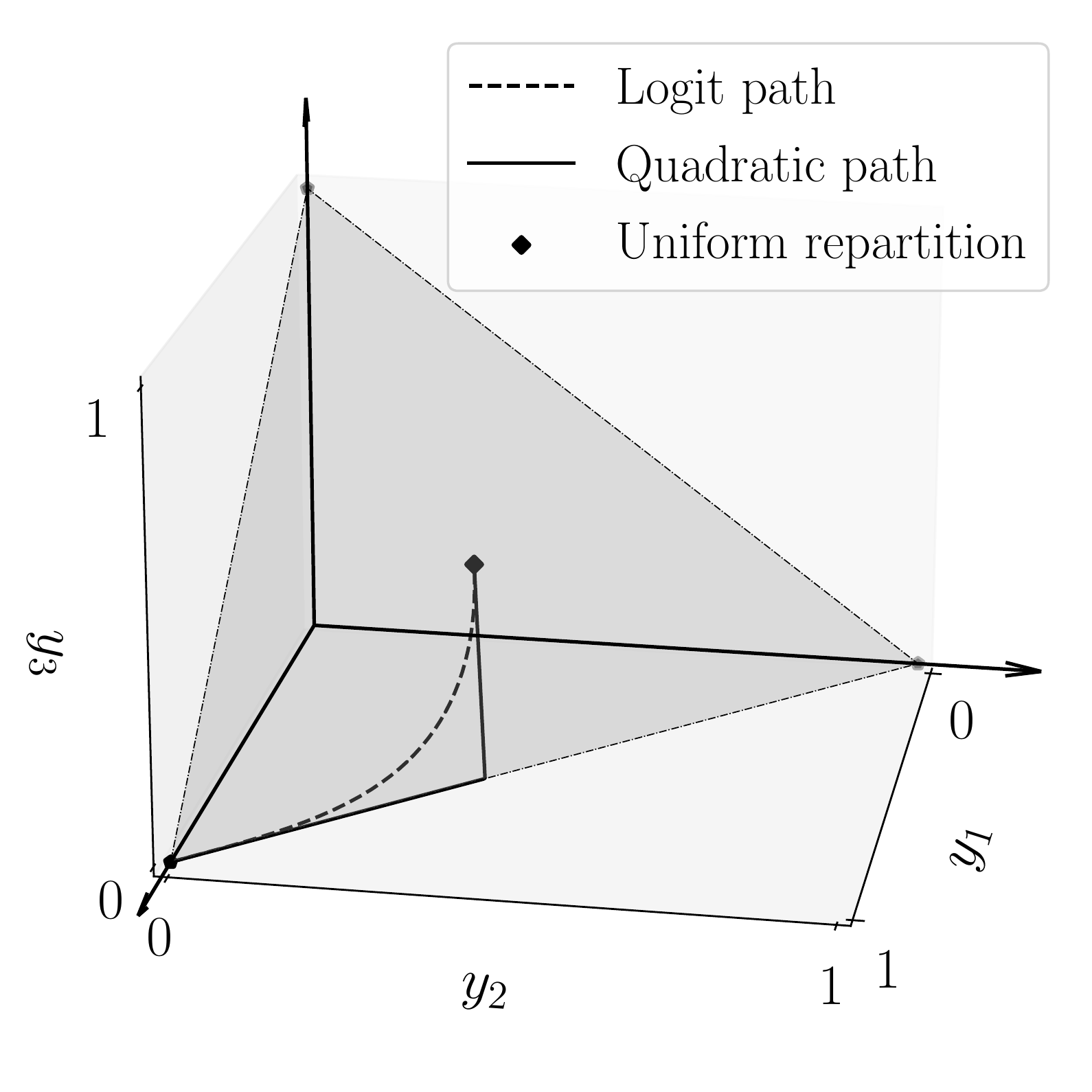}
    \caption{Logit and quadratic path on the simplex, as functions of $\beta$}
    \label{fig::simplex_quad_logit}
\end{figure}

\Cref{fig::simplex_quad_logit} illustrates \Cref{prop::metric_compare} and shows the logit and quadratic paths for a disutility vector $V=(0,\frac{1}{\sqrt{10}},\frac{3}{\sqrt{10}})$. The trajectory shares the same start point (the simplex center for $\beta=0$) and the same end point (the vertex $y=(1,0,0)$ for $\beta \rightarrow +\infty$). However, for the rest of the path the trajectories slightly deviate: we observe the sparsity effect of the projection operator in the behavior of the quadratic path whereas the logit trajectory always lies in the interior of the simplex.

\section{Performance analysis of the proposed method}\label{app::comparison_solvers}

The pre and post processing algorithms are implemented in \texttt{Python 3.7}, whereas the optimization methods are implemented in \texttt{C++} for numerical efficiency. Besides, we use \texttt{Cplex v12.10}~\shortciteA{CPLEX_2009v12} as a MIQP solver and the tests are performed on a laptop \texttt{Intel Core i7 @2.20GHz\,$\times$\,12}. We ran \texttt{Cplex} on $4$ threads.

\subsection{Comparison with implicit method}\label{sec::CMA}
Another way to solve the model \eqref{eq::bilevel_quad} is from the profit-maximization point of view, considering directly the nonsmooth problem 
\begin{equation}
\max_{x\in X} \pi^{quad}(x;\beta)
\label{eq::implicit_quad}
\end{equation} where the function $\pi^{quad}(\,\cdot\,;\beta)$ is defined in \eqref{eq::pi_quad}. 
Taking advantage of the lower response uniqueness to end in a nonsmooth problem -- where lower variables are functions of the upper ones -- constitutes the basis of implicit methods for bilevel problems, see~\shortciteA{Kim_2020}.

Implicit methods require an oracle able to evaluate the objective function for any given point. Therefore, the explicit calculation of the lower response given by \cref{corol::quad_lower_response} turns out to be essential in order to design the oracle.
Powerful algorithms are already available, and we focus on \emph{Covariance matrix adaptation evolution strategy} (\texttt{CMA-ES},~\shortciteA{Hansen_2006,Hansen_2010}). 
In our problem the search space $X$ has a reasonable dimension ($W \times H$).
Therefore, we can expect \texttt{CMA-ES} to find good solutions. For the numerical tests, we used an existing library available in \texttt{C++}\footnote{\url{https://github.com/CMA-ES/libcmaes}}.
\begin{figure}[!ht]
    \centering
    \includegraphics[width=0.85\linewidth]{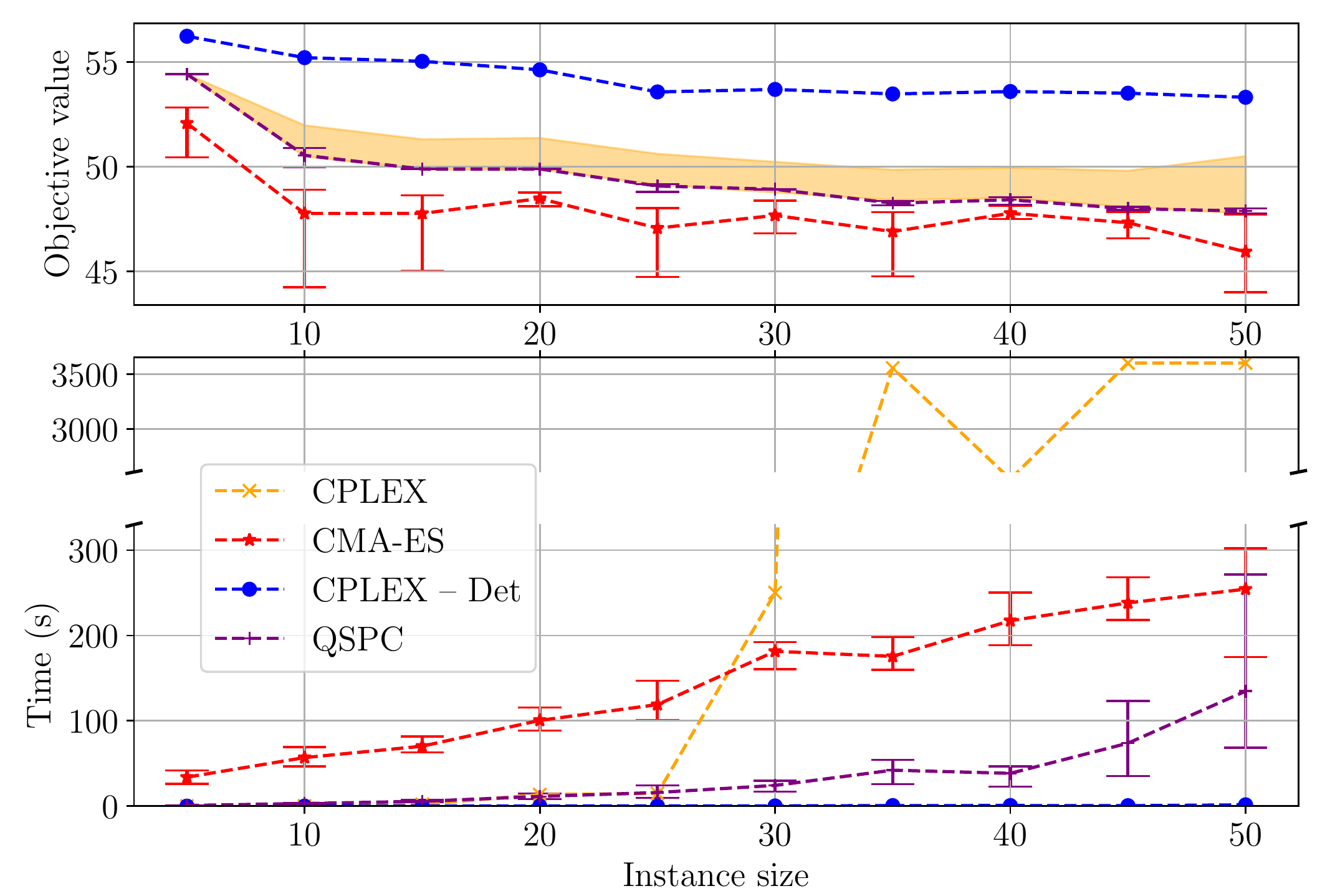}
    \caption{Numerical results with \texttt{CPLEX}, \texttt{CMA-ES}, \texttt{QSPC}.}
\medskip\small
The upper graph shows the objective value and the lower graph shows the resolution time for a segments number $S$ varying between 5 to 50 and a $\beta$ fixed to $0.5$. For heuristic methods, five tries have been done and vertical lines indicate the least and the greatest value. The final gap obtained with the quadratic method is represented with a yellow zone (between the best solution and the best upper bound).
For comparison, results of the deterministic model (\texttt{CPLEX - Det}) are given.
    \label{fig::result_S}
\end{figure}

\begin{table}[!ht]
\centering
\begin{tabular}{|c||c|c|c||c|}
\hline
 & \texttt{CMA-ES} & \texttt{CPLEX} & \texttt{QSPC} & \texttt{CPLEX - Det} \\
\hline\hline
Problem & \multicolumn{3}{c||}{$(q\beta\trt BP)$} & \multicolumn{1}{c|}{$(o\trt BP)$}\\
\hline
Method & \Cref{sec::CMA} & Eq.\eqref{eq::beta_MIQP} & \Cref{algo::QSPC} & \eqref{eq::o_LPCC}\\
\hline
\multirow{2}{*}{Parameters} & $\sigma = 0.005$ & MIP Gap $:3\%$ & $\sigma=0.05$ & \multicolumn{1}{c|}{MIP Gap$:1\%$}\\
 & $\lambda = 1000$ & Max time$:3600$s & $\gamma^S = \gamma^W = 1$ & \multicolumn{1}{c|}{Max time$:3600$s}\\
\hline
\end{tabular}
\caption{Methods used in the numerical tests}
\label{tab::method}
\end{table}
\Cref{fig::result_S} shows the performances of methods listed in \Cref{tab::method}.
The numerical tests highlight the combinatorial explosion induced by the direct resolution of the quadratic model with \texttt{CPLEX} for a finite $\beta$. The critical size seems to be around 30 segments on our data set. In contrast, the deterministic value is very fast to obtain up to 50 segments.
This emphasizes the need of heuristics to rapidly obtain good solutions of the quadratic model.

The method \texttt{CMA-ES} is rather suitable for very large instances. In fact, the algorithm explores the domain $X$ which does not depend on the number of segments $S$, and the time to compute the lower response
(by~\Cref{prop::construction_sol_quad}) linearly increases in $S$. The overall resolution time of \texttt{CMA-ES} has therefore an affine growth in the number of segments. Besides, the best solution found by \texttt{CMA-ES} seems to edge closer to optimum as the size grows. Increasing the number of segments dwindles the weight of each one in the objective, that tends to smooth the profit function and, as a consequence, facilitates \texttt{CMA-ES} in the resolution.

The great power of \texttt{QSPC} is to systematically find very good solutions (no large variance of the optimal value), even for large instances. Of course, this is only possible because we exploit the special geometry of our problem (as opposed to a generic algorithm like \texttt{CMA-ES}). Concerning resolution time, \texttt{QSPC} is also faster. However, \texttt{QSPC} becomes computationally more expensive as the number of segments increases, since it involves the restart phase
(solution of a MIQP problem).

Finally, this numerical study gives us an \textit{a posteriori} way to know how many segments are needed to accurately represent the population.  After 30 segments the objective value seems to reach a plateau: using more segments does not seem to add a useful information (at least in terms of optimal value).

\subsection{Comparison with NLP solvers}
Non-linear programming (NLP) constitutes a third alternative -- with implicit methods and combinatorial methods -- in the resolution of complementarity problems. Solvers have been designed/adapted to deal with these reformulations, see \shortciteA{Kim_2020} for a recent practical survey.  For the numerical tests, we focus on two solvers:
\begin{enumerate}[label=(\roman*)]
\setlength\itemsep{.1em}
  \item \texttt{KNITRO}~\shortciteA{Knitro_2021}, which is a powerful commercial solver, able to recognize if the problem contains complementarity constraints to reformulate them as a non-linear inequalities, 
  \item \texttt{filterMPEC}~\shortciteA{filterMPEC_NEOS}, which is an extension a Sequential Quadratic Programming (SQP) solver designed to solve MPECs. The theoretical material is described in~\shortciteA{Leyffer_2006}. Note that we keep the scalar product form (\texttt{compl\_frm = 1}) in all the resolutions.
\end{enumerate}
Both solvers are available through the platform \texttt{NEOS}~\shortciteA{NEOS}.

\begin{figure}[!ht]
    \centering
    \includegraphics[width=0.8\linewidth]{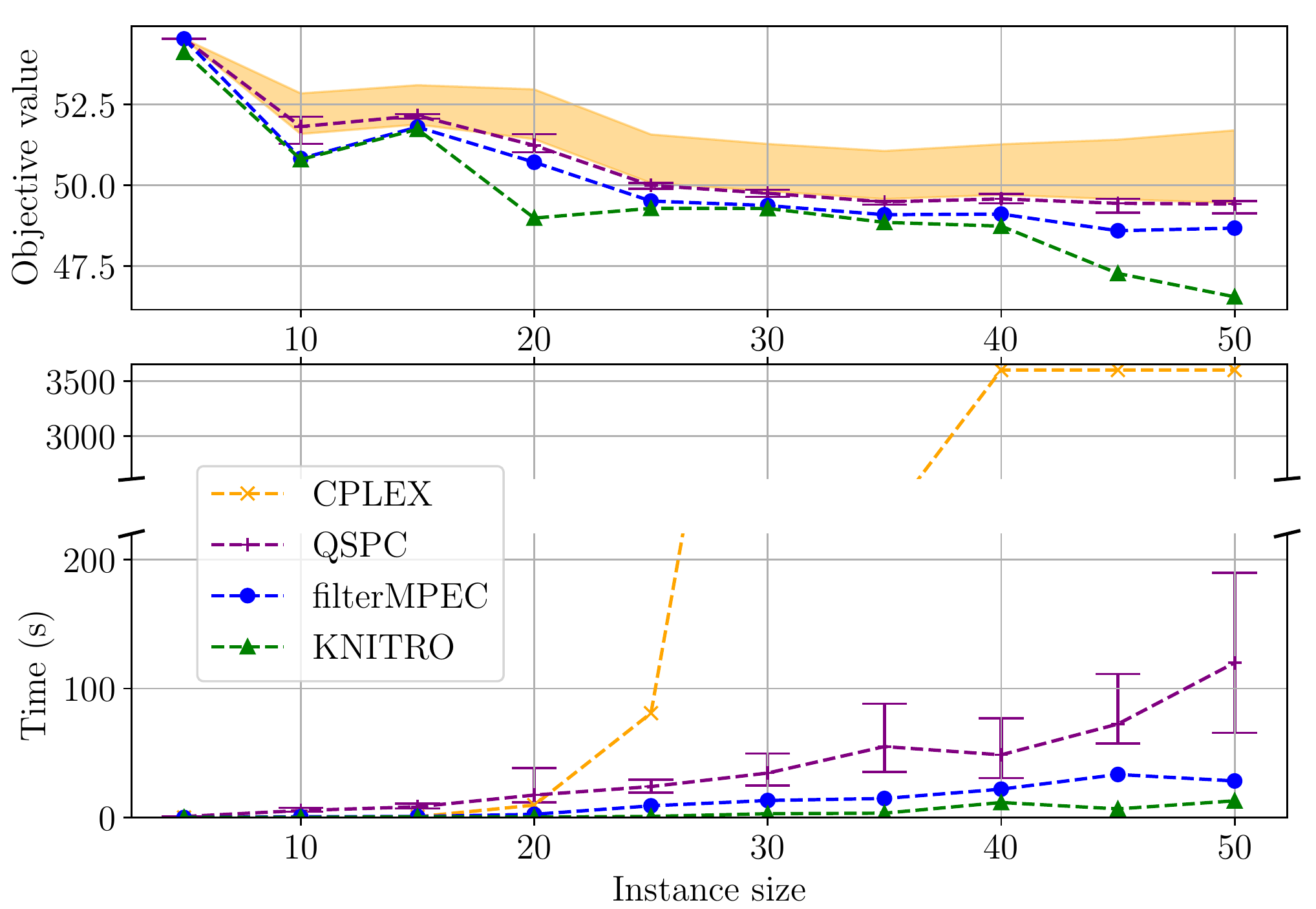}
    \caption{Comparison with NLP methods.}
   \medskip\small
The upper graph shows the objective value and the lower graph shows the resolution time for a segments number $S$ varying between 5 to 50 and a $\beta$ fixed to $0.5$. For heuristic methods, five tries have been done and vertical lines indicate the least and the greatest value. The final gap obtained with the quadratic method is represented with a yellow zone (between the best solution and the best upper bound).
    \label{fig::result_NL}
\end{figure}

\Cref{fig::result_NL} compares the results obtained by \texttt{KNITRO} and \texttt{filterMPEC} with our heuristic. We still display the value returned by \texttt{CPLEX} to bound the optimality gap. The whole graph is computed with instances that slightly differ from the ones on \Cref{fig::result_S}: the polytope $X$ only contains the bounds on prices and not any other constraint. In fact, the solution returned by NLP methods violates the constraints by an $\epsilon$ and if the polytope $X$ were more complicated than a box, it would require a finer post-processing to reconstruct a valid price vector $x$ that exactly respects the inequalities/equalities of $X$.

The two NLP solvers are very fast to return a solution, either \texttt{KNITRO} or \texttt{filterMPEQ}, even if the time cannot be considered as a uniform indicator since the calculations were achieved on \texttt{NEOS} servers whereas \texttt{QSPC} was run on a personal computer. On these instances, \texttt{QSPC} always returns better solutions.
In fact, only Clarke-Stationary points can be ensured by NLP solvers, see \shortciteA{Kim_2020} and the references therein. Of the two solvers, \texttt{KNITRO} seems to be the fastest, but we run it on 4 threads whereas \texttt{filterMPEC} uses a SQP algorithm which is difficult to parallelize. 

\end{document}